\documentclass[12pt,a4]{amsart}
\usepackage[a4paper, left=28mm, right=28mm, top=28mm, bottom=34mm]{geometry}
\usepackage{amsmath}
\usepackage{amsthm}
\usepackage{bm}
\usepackage[all]{xy}
\newtheorem{dfn}{Definition}[section]
\newtheorem{thm}[dfn]{Theorem}
\newtheorem{prop}[dfn]{Proposition}
\newtheorem{lem}[dfn]{Lemma}

\newtheorem{exm}[dfn]{Example}
\newtheorem{rem}[dfn]{Remark}

\usepackage{amssymb}
\usepackage{amscd}
\usepackage{mathrsfs}
\def\Ker{\mathop{\mathrm{Ker}}\nolimits}

\def\Q{\mathbb{Q}}
\def\R{\mathbb{R}}
\def\Z{\mathbb{Z}}
\def\C{\mathbb{C}}
\def\F{\mathbb{F}}

\def\Spec{\mathop{\mathrm{Spec}}\nolimits}
\def\Pic{\mathop{\mathrm{Pic}}\nolimits}
\def\hatBr{\mathop{\widehat{\mathrm{Br}}}\nolimits}

\def\g1{\mathop{\gamma_1}\nolimits}
\def\g2{\mathop{\gamma_2}\nolimits}
\def\O{\mathop{\mathscr{O}}\nolimits}

\def\disc{\mathop{\mathrm{disc}}}
\def\rank{\mathop{\mathrm{rank}}\nolimits}
\def\p{\mathfrak{p}}
\def\q{\mathfrak{q}}
\def\A{\mathbb{A}}
\def\Af{\mathbb{A}_f}
\def\`e{\mathrm{\grave{e}}}
\def\'e{\mathrm{\acute{e}}}
\def\GL{\mathrm{GL}}

\title[Supersingular reduction of K3 surfaces with CM]{On the supersingular reduction of $K3$ surfaces with complex multiplication}
\address{Department of Mathematics, Faculty of Science, Kyoto University, Kyoto 606-8502, Japan}
\email{kito@math.kyoto-u.ac.jp}
\date{\today}
\author{Kazuhiro Ito}

\begin{document}
\maketitle

\begin{abstract}
We study the good reduction modulo $p$ of $K3$ surfaces with complex multiplication. If a $K3$ surface with complex multiplication has good reduction, we calculate the Picard number and the height of the formal Brauer group of the reduction. Moreover, if the reduction is supersingular, we calculate its Artin invariant under some assumptions. Our results generalize some results of Shimada for $K3$ surfaces with Picard number $20$. Our methods rely on the main theorem of complex multiplication for $K3$ surfaces by Rizov, an explicit description of the Breuil-Kisin modules associated with Lubin-Tate characters due to Andreatta, Goren, Howard, and Madapusi Pera, and the integral comparison theorem recently established by Bhatt, Morrow, and Scholze. 
\end{abstract}

\maketitle

\section{Introduction}\label{Introduction}

Let $X_{\C}$ be a projective $K3$ surface over $\C$. Recall that a \textit{$K3$ surface} $X$ over a field is a projective smooth surface with trivial canonical bundle satisfying $H^1(X,\O_X)=0$. Let
$$T(X_{\C}):=\Pic(X_{\C})^{\perp} \subset H^{2}(X_{\C}, \Z(1))$$
be the transcendental part of the singular cohomology, which has the $\Z$-Hodge structure coming from $H^{2}(X_{\C}, \Z(1))$. Let $$\mathrm{End}_{\mathrm{Hdg}}(T(X_{\C}))$$ be the $\Z$-algebra of $\Z$-linear endomorphisms on $T(X_{\C})$ preserving the $\Z$-Hodge structure on it.

Let $F$ be a totally real number field. Let $E$ be a totally imaginary quadratic extension of $F$. (Such field $E$ is called a \textit{$CM$ field}.)  The $K3$ surface $X_{\C}$ is said to have \textit{complex multiplication} $(CM)$ by $E$ if the $\Q$-algebra 
\[
E(X_{\C}):=\mathrm{End}_{\mathrm{Hdg}}(T(X_{\C}))\otimes_{\Z}{\Q}
\]
is isomorphic to $E$ and $T(X_{\C})\otimes_{\Z}{\Q}$ is a one-dimensional $E(X_{\C})$-vector space. 
In the rest of this introduction, 
we fix a projective $K3$ surface $X_\C$ over $\C$ 
with $CM$ by $E$.
We fix an isomorphism $E(X_{\C})\simeq E$. Then, the canonical morphism
\[
\epsilon_{X_{\C}}\colon E(X_{\C})\rightarrow \mathrm{End}_{\C}(H^{2, 0}(X_{\C}))\simeq \C
\]
gives an embedding $E \hookrightarrow \C$. By this embedding, we consider $E$ as a subfield of $\C$. It is known that $X_{\C}$ has a model $X_K$ over a number field $K \subset \C$ which contains $E$; see \cite[Theorem 4]{Shafarevich2} and \cite[Corollary 3.9.4]{Rizov10}. 

We fix a prime number $p$ and a finite place $v$ of $K$ whose residue characteristic is $p$. 
Recall that we have extensions of fields
\[
\Q \subset F \subset E \subset K \subset \C
\]
with $[E:F]=2$.
Let $\p$, $\q$ be the finite places of $E$, $F$ below $v$, respectively. We denote the completion of $K$ at $v$ by $K_v$, the valuation ring of $K_v$ by $\mathscr{O}_{K_v}$, the residue field of $K_v$ by $k(v)$, and an algebraic closure of $k(v)$ by $\overline{k(v)}$. Similarly, we use the same notation $E_{\p}$, $\O_{E_{\p}}$, $k(\p)$, $F_{\q}$, $\O_{F_{\q}}$, $k(\q)$ for $\p$ and $\q$. 

It is conjectured that $X_{K}$ has \textit{potential good reduction} at $v$, i.e.\ after replacing $K_v$ by a finite extension of it, there exists a proper smooth algebraic space $\mathscr{X}$ over $\Spec\mathscr{O}_{K_v}$ whose generic fiber 
$
\mathscr{X} \otimes_{\mathscr{O}_{K_v}} K_v
$
is isomorphic to $X_{K_v}:=X_K\otimes_{K}{K_v}$.
For sufficient conditions for $X_{K}$ to have potential good reduction at $v$, see Remark \ref{good reduction}.
 In this paper, we assume that there exists such an algebraic space $\mathscr{X}$ over $\Spec\mathscr{O}_{K_v}$, and study the geometric special fiber of $\mathscr{X}$, which is denoted by 
\[
\mathscr{X}_{\overline{v}}:=\mathscr{X}\otimes_{\mathscr{O}_{K_v}}{\overline{k(v)}}.
\]
We remark that $\mathscr{X}_{\overline{v}}$ is a $K3$ surface over ${\overline{k(v)}}$.

Our first main result is the calculation of the Picard number of $\mathscr{X}_{\overline{v}}$ and the height of the formal Brauer group of $\mathscr{X}_{\overline{v}}$. (For the formal Brauer groups of $K3$ surfaces in characteristic $p$, see Section \ref{Formal Brauer group}.)
\begin{thm}\label{supersingular reduction}
\begin{enumerate}
\item[(1)] If $\q$ splits in $E$, the Picard number of $\mathscr{X}_{\overline{v}}$ is equal to $22-[E:\Q]$, and the height of the formal Brauer group of $\mathscr{X}_{\overline{v}}$ is $[E_{\p}:\Q_p]$.
\item[(2)] If $\q$ does not split in $E$
(i.e.\ $\q$ is either inert or ramified in $E$),
the Picard number of $\mathscr{X}_{\overline{v}}$ is $22$, and the height of the formal Brauer group of $\mathscr{X}_{\overline{v}}$ is $\infty$. In this case, \ $\mathscr{X}_{\overline{v}}$ is a supersingular $K3$ surface over $\overline{k(v)}$. 
\end{enumerate}
\end{thm}

Since $\dim_{\Q}T(X_{\C})\otimes_{\Z}{\Q}=[E:\Q]$, the Picard number of $X_{\C}$ is equal to $22-[E:\Q]$. Theorem \ref{supersingular reduction} shows if $\q$ splits in $E$, the Picard number does not jump in specialization. On the other hand, if $\q$ does not split in $E$, the Picard number always jumps in specialization.

Assume that $\q$ does not split in $E$. Then $\Pic(\mathscr{X}_{\overline{v}})$ endowed with the intersection product is a lattice of rank $22$ by Theorem \ref{supersingular reduction} (2). The determinant of the matrix representation of the intersection product is denoted by $\disc \Pic(\mathscr{X}_{\overline{v}})$. We call it the \textit{discriminant} of $\Pic(\mathscr{X}_{\overline{v}})$. It is known that $\disc \Pic(\mathscr{X}_{\overline{v}})$ is of the form $-p^{2a}$ for an integer $1\leq a \leq10$; see \cite[Proposition 1 in Section 8]{Rudakov-Shafarevich}, \cite[Theorem in Section 8]{Rudakov-Shafarevich}, and \cite[Chapter 17, Proposition 2.19]{Huybrechts}. The integer $a$ is called the \textit{Artin invariant} of the supersingular $K3$ surface $\mathscr{X}_{\overline{v}}$. 

Our second main result is the calculation of the Artin invariant of $\mathscr{X}_{\overline{v}}$ under some assumptions on the discriminant of the Picard group and the endomorphism algebra of the Hodge structure. 
\begin{thm}\label{Artin invariant}
Assume that $\q$ does not split in $E$, i.e.\ the $K3$ surface $\mathscr{X}_{\overline{v}}$ is supersingular. Moreover, we assume that both of the following conditions hold:
\begin{itemize}
\item The discriminant $\disc \Pic(X_{\C})$ of the lattice $\Pic(X_{\C})$ is not divisible by $p$.
\item The isomorphism $E(X_{\C}) \simeq E$ induces an isomorphism
$$
\mathrm{End}_{\mathrm{Hdg}}(T(X_{\C}))\otimes_{\Z}{\Z_{(p)}}\simeq \mathscr{O}_{E}\otimes_{\Z}{\Z_{(p)}},
$$ 
where $\mathscr{O}_{E}$ is the ring of integers of $E$.
\end{itemize}
Then the Artin invariant of $\mathscr{X}_{\overline{v}}$ is equal to $[k(\q):\F_{p}]$.
\end{thm}

The assumptions of Theorem \ref{Artin invariant} hold for all but finitely many prime numbers $p$. It is well-known that the first assumption does not always hold. For $K3$ surfaces which do not satisfy the second assumption, see Example \ref{counter}.
In Example \ref{counter}, we also show that the conclusion of Theorem \ref{Artin invariant} is not true in general.

\begin{rem}\label{elliptic curve reduction}
\rm{Theorem \ref{supersingular reduction} is an analogue of a well-known result for elliptic curves with $CM$; an elliptic curve $C$ over $\C$ which has $CM$ by an imaginary quadratic field $E$ is defined over a number field $K$, it has potential good reduction at every finite place of $K$, and the good reduction of $C$ in characteristic $p>0$ is ordinary (resp.\ supersingular) if and only if $p$ splits (resp.\ does not split) in $E$. 
(See \cite[Chapter 13, Theorem 12]{Lang} for example.)
}

\end{rem}

\begin{rem}
\rm{Charles studied the behavior of specialization at a finite place of 
the Picard groups of K3 surfaces over number fields \cite[Theorem 1]{Charles14}.
}

\end{rem}

\begin{rem}
\rm{A projective $K3$ surface $X_{\C}$ over $\C$ with Picard number $20$ has $CM$ by an imaginary quadratic field $E$; see \cite[Chapter 3, Remark 3.10]{Huybrechts}. Hence $X_{\C}$ has a model $X_K$ over a number field $K$. Shimada determined when the reduction of $X_K$ in characteristic $p$ is supersingular and proved that its Artin invariant is equal to $1$ under the assumption that $p$ does not divide $\disc \Pic(X_{\C})$ and $p\neq2$; see \cite[Proposition 5.5]{Shimada} and \cite[Theorem 1]{Shimada2}. In Proposition \ref{singular3}, we shall show that Shimada's assumptions imply the assumptions of Theorem \ref{Artin invariant}. Therefore, our results imply the results of Shimada. Our proofs and Shimada's proofs are completely different. Note that we do not exclude the case $p=2$ in this paper. By applying our results, we can prove that the results of Shimada also hold in the case $p=2$; see Theorem \ref{Shimada2} for details. We also remark that Sch${\mathrm{\overset{..}{u}}}$tt showed that if $p$ is inert in $E$, the reduction of $X_K$ in characteristic $p$ is supersingular, where $p$ is any prime number (including $p=2$); see \cite[Proposition 4.1]{Schutt}. 
}
\end{rem}

\begin{rem}
\rm{A projective $K3$ surface $X_{\C}$ over $\C$ which has an automorphism $f \in \mathrm{Aut}(X_{\C})$ such that the order of the induced map $f^{*} \in \mathrm{Aut}(T(X_{\C}))$ is $N$ with $\phi(N)=\rank_\Z{T(X_{\C})}$ has $CM$ by the cyclotomic field $\Q(\zeta_{N})$, and we have 
$\mathrm{End}_{\mathrm{Hdg}}(T(X_\C))\simeq \mathscr{O}_{\Q(\zeta_{N})}
=\Z[\zeta_{N}]$.
Here $\zeta_{N}\in\C$ is a primitive $N$-th root of unity and $\phi$ is Euler's totient function.
For such a $K3$ surface $X_{\C}$, Theorem \ref{supersingular reduction} and 
Theorem \ref{Artin invariant} were proved by Jang when $p \neq 2$; see \cite[Theorem 2.3]{Jang14}, \cite[Corollary 3.3]{Jang16}.}

\end{rem}

\begin{rem}\label{good reduction}
\rm{For a projective $K3$ surface $X_{K}$ over a number field $K \subset \C$ such that $X_{\C}:=X_{K}\otimes_{K}{\C}$ has $CM$, it is conjectured that $X_{K}$ has potential good reduction at every finite place $v$ of $K$. Let $v$ be a finite place of $K$ with residue characteristic $p$. In \cite{Matsumoto}, Matsumoto proved that $X_{K_v}$ has potential good reduction when $X_{K_v}$ satisfies at least one of the following conditions:
\begin{itemize}
\item $X_{K_v}$ admits an ample line bundle $\mathscr{L}$ with $p>(\mathscr{L})^2+4$.
\item $p\geq5$ and $X_{K_v}$ admits an elliptic fibration with a section. 
\item $p\geq5$ and the Picard number $\rho(X_{\overline{K_{v}}})$ of $X_{\overline{K_{v}}}$ satisfies $12 \leq \rho(X_{\overline{K_{v}}})\leq 20.$
\end{itemize}
For details, see \cite[Theorem 6.3]{Matsumoto} and its proof.
(See also \cite[Theorem 2.5]{Ito}.)
}
\end{rem}

Let us explain the outline of the proof of Theorem \ref{supersingular reduction} and Theorem \ref{Artin invariant}. 

We first prove Theorem \ref{supersingular reduction} using the main theorem of complex multiplication for $K3$ surfaces proved by Rizov \cite[Corollary 3.9.2]{Rizov10}. (Similar arguments can be found in Taelman's paper; see the proof of \cite[Proposition 25]{Taelman}.)
 
For the proof of Theorem \ref{Artin invariant}, we shall give an explicit form of the $F$-crystal 
\[
H^2_{\mathrm{cris}}(\mathscr{X}_{\overline{v}}/W)
\]
 under the assumptions of Theorem \ref{Artin invariant}, where $W:=W(\overline{\F}_{p})$ is the ring of Witt vectors of $\overline{k(v)}=\overline{\F}_{p}$. Then we compute the cokernel of the Chern class map 
$${\mathrm{ch_{cris}}}\colon \Pic(\mathscr{X}_{\overline{v}}){\otimes}_{\Z}{W} \rightarrow H^2_{\mathrm{cris}}(\mathscr{X}_{\overline{v}}/W)$$ 
to calculate the Artin invariant.

To describe the $F$-crystal $H^2_{\mathrm{cris}}(\mathscr{X}_{\overline{v}}/W)$ explicitly, we describe the Galois module 
$$
H^2_{\mathrm{\'et}}(X_{\overline{K_v}}, \Z_{p}),
$$
after replacing $K_{v}$ by a finite extension of it. It is a $\Z_p$-lattice in the crystalline representation $H^2_{\rm{\acute{e}t}}(X_{\overline{K_v}}, \Q_{p}).$  Here we use the main theorem of complex multiplication for $K3$ surfaces again.
According to the integral comparison theorem recently established by Bhatt, Morrow, and Scholze \cite[Theorem 1.2]{BMS0}, \cite[Theorem 14.6]{BMS}, the $F$-crystal $H^2_{\mathrm{cris}}(\mathscr{X}_{\overline{v}}/W)$ can be recovered from the Galois module $H^2_{\rm{\acute{e}t}}(X_{\overline{K_v}}, \Z_{p})$. More precisely, to calculate the $F$-crystal $H^2_{\mathrm{cris}}(\mathscr{X}_{\overline{v}}/W)$, it is enough to calculate the Breuil-Kisin module
 $$
 \mathfrak{M}(H^2_{\rm{\acute{e}t}}(X_{\overline{K_v}}, \Z_{p}))
 $$
  associated with the Galois module $H^2_{\rm{\acute{e}t}}(X_{\overline{K_v}}, \Z_{p})$.     
The main ingredient for the calculation of the Breuil-Kisin module is an explicit description of the Breuil-Kisin modules associated with Lubin-Tate characters proved by Andreatta, Goren, Howard, and Madapusi Pera \cite[Proposition 2.2.1]{AGHM}.

\begin{rem}
\rm{The Breuil-Kisin modules of $K3$ surfaces are also studied by Chiarellotto, Lazda, and Liedtke in \cite{CCL}.}
\end{rem}

The outline of this paper is as follows. In Section \ref{CMsection}, we recall the definition of complex $K3$ surfaces with $CM$ and the main theorem of complex multiplication for $K3$ surfaces due to Rizov. In Section \ref{Formal Brauer group}, we recall the definition and basic properties of the formal Brauer groups of $K3$ surfaces in characteristic $p>0$. In Section \ref{Section_proof_thm_1}, we prove Theorem \ref{supersingular reduction}. In Section \ref{The Galois module} and Section \ref{F-crystal associated with Lubin-Tate characters}, we make some preparations to prove Theorem \ref{Artin invariant}. In Section \ref{The Galois module}, we give an explicit description of the Galois module $H^2_{\rm{\acute{e}t}}(X_{\overline{K_v}}, \Z_{p})$ in terms of Lubin-Tate characters. In Section \ref{F-crystal associated with Lubin-Tate characters}, we recall the results of Andreatta, Goren, Howard, and Madapusi Pera on an explicit description of the Breuil-Kisin modules associated with Lubin-Tate characters, and calculate the Breuil-Kisin module associated with the Galois module $H^2_{\rm{\acute{e}t}}(X_{\overline{K_v}}, \Z_{p})$. Then, we give an explicit description of the $F$-crystal $H^2_{\mathrm{cris}}(\mathscr{X}_{\overline{v}}/W)$ by applying the integral comparison theorem due to Bhatt, Morrow, and Scholze. In Section \ref{Section_proof_thm2}, we prove Theorem \ref{Artin invariant} using the results of Section \ref{F-crystal associated with Lubin-Tate characters}. In Section \ref{singular K3}, we prove our results imply the results of Shimada on $K3$ surfaces with Picard number $20$. Moreover, we prove that the results of Shimada also hold in the case $p=2$. Finally, in Section \ref{examples}, we show that there are projective $K3$ surfaces over $\C$ with $CM$ which do not satisfy the assumptions and the conclusion of Theorem \ref{Artin invariant}.

\section{Complex multiplication for $K3$ surfaces}\label{CMsection}
In this section, we recall the definition of projective $K3$ surfaces over $\C$ with complex multiplication $(CM)$ and the main theorem of complex multiplication for $K3$ surfaces due to Rizov \cite{Rizov10}.

Let $X_{\C}$ be a projective $K3$ surface over $\C$. Let 
\[
T(X_{\C}):=\Pic(X_{\C})^{\perp} \subset H^{2}(X_{\C}, \Z(1))
\]
be the transcendental part of the singular cohomology, which has the $\Z$-Hodge structure coming from $H^{2}(X_{\C}, \Z(1))$. Let $$\mathrm{End}_{\mathrm{Hdg}}(T(X_{\C}))$$ be the $\Z$-algebra of $\Z$-linear endomorphisms on $T(X_{\C})$ preserving the $\Z$-Hodge structure on it. Zarhin showed that the $\Q$-algebra
\[
E(X_{\C}):=\mathrm{End}_{\mathrm{Hdg}}(T(X_{\C}))\otimes_{\Z}{\Q}
\]
is a number field which is either totally real or $CM$; see \cite[Theorem 1.5.1]{Zarhin}, \cite[Theorem 1.6(a)]{Zarhin}. Here a number field is called $CM$ if it is a totally imaginary quadratic extension of a totally real number field. Let $E$ be a $CM$ field. We say $X_{\C}$ has \textit{complex multiplication $(CM)$} by $E$ if $E(X_{\C})$ is isomorphic to $E$ and $T({X_{\C}})\otimes_{\Z}{\Q}$ is a one-dimensional $E(X_{\C})$-vector space. In this section, we assume that $X_{\C}$ has $CM$ by $E$, and we fix an isomorphism $E(X_{\C})\simeq E$. Then, the canonical morphism
\[
\epsilon_{X_{\C}}\colon E(X_{\C})\rightarrow \mathrm{End}_{\C}(H^{2, 0}(X_{\C}))\simeq \C
\]
gives an embedding $E \hookrightarrow \C$. By this embedding, we consider $E$ as a subfield of $\C$.

Let $G:=\mathrm{MT}(T({X_{\C}})\otimes_{\Z}{\Q})$ be the \textit{Mumford-Tate group} of $T({X_{\C}})\otimes_{\Z}{\Q}$; see Section \ref{examples} for details.
Since the $\Q$-Hodge structure on $T({X_{\C}})\otimes_{\Z}{\Q}$ is of weight $0$, $G$ is isomorphic to the \textit{Hodge group} of $T({X_{\C}})\otimes_{\Z}{\Q}(-1)$; see Remark \ref{MT and Hdg}.
Zarhin proved that $G$ is isomorphic to 
\[
\Ker(\mathrm{Nm}\colon \mathrm{Res}_{E/\Q}\mathbb{G}_m\rightarrow \mathrm{Res}_{F/\Q}\mathbb{G}_m)
\]
as an algebraic group over $\Q$ via $E(X_{\C}) \simeq E$, where $F$ is the maximal totally real subfield of $E$; see \cite[Theorem 2.3.1]{Zarhin}. Hence we have
\[
G(\Q)=\{\, x \in{E^{\times}} \mid xc(x)=1 \, \},
\]
where $c\colon E \rightarrow E$ is the complex conjugation. We also denote the involution on the ring $\A_{E, f}$ of finite ad$\`e$les of $E$ induced by $c$ by the same symbol $c$.

\rm{Pjatecki{\u\i}-{\v{S}}apiro and {\v{S}}afarevi{\v{c}} showed that $X_{\C}$ has a model $X_K$ over a number field $K \subset \C$ which contains $E$ \cite[Theorem 4]{Shafarevich2}. (Rizov showed that we can take such a number field $K \subset \C$ as an abelian extension of $E$ \cite[Corollary 3.9.4]{Rizov10}. However we will not use this fact in this paper.) We fix an embedding $\overline{K}\hookrightarrow \C$. We have an action of $\mathrm{Gal}(\overline{K}/K)$ on the $\'e$tale  cohomology
\[
H^2_{\rm{\acute{e}t}}(X_{\overline{K}}, \mathbb{A}_{f}{(1)}):=(\prod_{p}H^2_{\rm{\acute{e}t}}(X_{\overline{K}}, \Z_p(1)))\otimes_{\Z}{\Q},
\]
which is a free module of rank $22$ over the ring $\A_f:=(\prod_{p}\Z_p)\otimes_{\Z}{\Q}$ of finite ad$\`e$les of $\Q$. The embedding $\overline{K}\hookrightarrow \C$ gives an isomorphism
\[
H^2_{\rm{\acute{e}t}}(X_{\overline{K}}, \mathbb{A}_{f}{(1)}) \simeq H^{2}(X_{\C}, \Z(1))\otimes_{\Z}{\mathbb{A}_f}.
\]
It induces an action of $\mathrm{Gal}(\overline{K}/K)$ on $T({X_{\C}})\otimes_{\Z}\Af$. 

Now we recall the main theorem of complex multiplication for $X_{\C}$ due to Rizov \cite{Rizov10}. Let 
\[
{\mathrm{Art}_{K}}\colon\A^{\times}_{K}\rightarrow \mathrm{Gal}(K^{\mathrm{ab}}/K) 
\] 
be the Artin reciprocity map provided by global class field theory with Deligne's normalization, i.e.\ which sends uniformizers to lifts of  geometric Frobenius elements. Recall that we have $E \subset K$ as subfields of $\C$. Let 
\[
\mathrm{Nm}_{\A_K/\A_E}\colon\A^{\times}_{K} \rightarrow \A^{\times}_{E}
\]
be the norm map induced by the embedding $E \subset K$.
The main theorem of complex multiplication for $K3$ surfaces describes the action of $\mathrm{Gal}(\overline{K}/K)$ on $T({X_{\C}})\otimes_{\Z}\Af$ in terms of the norm map $\mathrm{Nm}_{\A_K/\A_E}$ and the Artin reciprocity map ${\mathrm{Art}_{K}}$.

\begin{thm}[Rizov {\cite{Rizov10}}]\label{CM main theorem}
The action of $\mathrm{Gal}(\overline{K}/K)$ on $T({X_{\C}})\otimes_{\Z}\Af$ factors through a continuous homomorphism 
\[
\rho\colon \mathrm{Gal}(K^{\mathrm{ab}}/K) \rightarrow G(\Af)=\{\, x \in{\A^{\times}_{E, f}} \mid xc(x)=1 \, \}.
\] 
Moreover, the following diagram commutes:
\[
\xymatrix{
\A^{\times}_{K}\ar[r]_-{\mathrm{Nm}_{\A_K/\A_E}}\ar[d]_{\mathrm{Art}_{K}}&\A^{\times}_{E}\ar[d]^{\mathrm{pr}}\\
\mathrm{Gal}(K^{\mathrm{ab}}/K)\ar[d]_{\rho}&\A^{\times}_{E, f}\ar[d]^{y\mapsto c(y)y^{-1}} \\
G(\Af) \ar[r] & G(\Af)/G(\Q).
}
\] 
\end{thm}

\begin{proof}
See \cite[Corollary 3.9.2]{Rizov10}. See also \cite[Theorem 12]{Taelman}. 
\end{proof}

We fix a prime number $p$ and a finite place $\p$ of $E$ whose residue characteristic is $p$. Let $\q$ be the finite place of $F$ below $\p$. In the rest of this section, we shall show that the extension $E_{\p}/F_{\q}$ is unramified under the assumptions of Theorem \ref{Artin invariant}. We first make some preparations.

\begin{lem}\label{quadratic space}
We assume that the isomorphism $E(X_{\C}) \simeq E$ induces an isomorphism
\[
\mathrm{End}_{\mathrm{Hdg}}(T(X_{\C}))\otimes_{\Z}{\Z_{(p)}}\simeq \mathscr{O}_{E}\otimes_{\Z}{\Z_{(p)}}. 
\]
Then $T(X_{\C})\otimes_{\Z}{\Z_{(p)}}$ is a free $\mathscr{O}_{E}\otimes_{\Z}{\Z_{(p)}}$-module of rank $1$. Moreover, if we fix an isomorphism of $\mathscr{O}_{E}\otimes_{\Z}{\Z_{(p)}}$-modules
\[
T(X_{\C})\otimes_{\Z}{\Z_{(p)}}\simeq\mathscr{O}_{E}\otimes_{\Z}{\Z_{(p)}}, 
\]
there exists an element $d \in F^{\times}$ such that the intersection product on $T(X_{\C})\otimes_{\Z}{\Z_{(p)}}$ is translated to the pairing 
\[
\xymatrix{
\langle\ ,\ \rangle\colon\mathscr{O}_{E}\otimes_{\Z}{\Z_{(p)}}\times\mathscr{O}_{E}\otimes_{\Z}{\Z_{(p)}} \ar[r]& \Z_{(p)}}
\]
defined by 
$
\langle x, y \rangle=\mathrm{Tr}_{E/\Q}(dxc(y))
$
for every $x, y \in \mathscr{O}_{E}\otimes_{\Z}{\Z_{(p)}}$.
\end{lem}
\begin{proof}
Since $\mathscr{O}_{E}\otimes_{\Z}{\Z_{(p)}}$ is a principal ideal domain and $T({X_{\C}})\otimes_{\Z}{\Q}$ is a one-dimensional $E$-vector space, it follows that $T(X_{\C})\otimes_{\Z}{\Z_{(p)}}$ is a free $\mathscr{O}_{E}\otimes_{\Z}{\Z_{(p)}}$-module of rank $1$. The second assertion is well-known; see the proof of \cite[Proposition 1.3.10]{Chen} for example. 
\end{proof}

\begin{lem}\label{norm}
Let $L_1$ be a field which is a finite extension of the field $\Q_2$ of $2$-adic numbers. Let $L_2$ be a totally ramified quadratic extension of $L_1$. Then $\O_{L_1}$ is generated by $\mathrm{Nm}_{L_2/L_1}(\O_{L_2})$ as a $\Z_2$-module.
\end{lem}
\begin{proof}
Let $\pi_{L_2}$ be a uniformizer of $L_2$. We put 
$\pi_{L_1} := \mathrm{Nm}_{L_2/L_1}(\pi_{L_2})$. 
Since the extension $L_2/L_1$ is totally ramified, the element $\pi_{L_1}$ is a uniformizer of $L_1$.

Since the sub-$\Z_2$-module of $\O_{L_1}$ generated by $\mathrm{Nm}_{L_2/L_1}(\O_{L_2})$ is closed in $\O_{L_1}$, it suffices to show $\O_{L_1}/\pi_{L_1}^{n}\O_{L_1}$ is generated by the image of $\mathrm{Nm}_{L_2/L_1}(\O_{L_2})$ as a sub-$\Z_2$-module for every $n \geq 1$. Since $\pi_{L_1} \in \mathrm{Nm}_{L_2/L_1}(\O_{L_2})$ and $\mathrm{Nm}_{L_2/L_1}(\O_{L_2})$ is closed under multiplication, it suffices to show $\O_{L_1}/\pi_{L_1}\O_{L_1}$ is generated by the image of $\mathrm{Nm}_{L_2/L_1}(\O_{L_2})$ as a sub-$\Z_2$-module. 

We put $k:=\O_{L_1}/\pi_{L_1}\O_{L_1}$. Since the extension $L_2/L_1$ is totally ramified, the image of $\mathrm{Nm}_{L_2/L_1}(\O_{L_2})$ in $k$ coincides with the image of the Frobenius map $k \rightarrow k, x \mapsto x^2$. Since $k$ is a finite field of characteristic $2$, we see that the image of $\mathrm{Nm}_{L_2/L_1}(\O_{L_2})$ in $k$ is $k$. 

The proof of Lemma \ref{norm} is complete.
\end{proof}

We conclude this section by proving the following proposition.
\begin{prop}\label{local_lemma}
Under the assumptions of Theorem \ref{Artin invariant}, the quadratic extension $E_{\p}/F_{\q}$ is unramified.
\end{prop}
\begin{proof}
By Lemma \ref{quadratic space}, we may identify $T(X_{\C})\otimes_{\Z}{\Z_{(p)}}$ with $\mathscr{O}_{E}\otimes_{\Z}{\Z_{(p)}}$, where $\mathscr{O}_{E}\otimes_{\Z}{\Z_{(p)}}$ is equipped with a pairing $\langle\ ,\ \rangle$ defined by 
$
\langle x, y\rangle=\mathrm{Tr}_{E/\Q}(dxc(y))
$
for every $x, y \in \mathscr{O}_{E}\otimes_{\Z}{\Z_{(p)}}$ for some element $d \in F^{\times}$. By the first assumption of Theorem \ref{Artin invariant}, the discriminant of $\Pic(X_{\C})$ is not divisible by $p$. Hence the pairing $\langle\ ,\ \rangle$ is perfect. There is a decomposition \[
\mathscr{O}_{E}\otimes_{\Z}{\Z_p} \simeq \prod_{\p'|p}\mathscr{O}_{E_{\p'}},
\]
where $\p'$ is a finite place of $E$ above $p$. Using this decomposition, we have a perfect pairing
\[
\xymatrix{
\mathscr{O}_{E_{\p}}\times\mathscr{O}_{E_{\p}} \ar[r]& \Z_{p}}
\]
\[
\xymatrix{(x, y) \ar@{|->}[r]  & \mathrm{Tr}_{E_{\p}/\Q_p}(dxc(y)),
}
\]
where $c \colon E_{\p} \simeq E_{\p}$ is the involution induced by $c$ on $E$. We denote this pairing by $\langle\ ,\ \rangle_{\p}$.

Let $\mathcal{D}_{E_{\p}/\Q_p}$ be the different of $E_{\p}$ over $\Q_p$. Since $\langle\ ,\ \rangle_{\p}$ is a perfect pairing, we see that $\mathcal{D}_{E_{\p}/\Q_p}$ is generated by $d^{-1}$ as an ideal of $\O_{E_{\p}}$.

Suppose that $p>2$. Let $\nu_{\p}\colon E_{\p}^{\times} \rightarrow \Z$ be the valuation normalized by $\nu_{\p}(\pi_{\p})=1$, where $\pi_{\p}$ is a uniformizer of $E_{\p}$. Let $\mathcal{D}_{E_{\p}/F_{\q}}$ be the different of $E_{\p}$ over $F_{\q}$. We assume that $E_{\p}/F_{\q}$ is totally ramified. Then, since $p>2$, a generator $e\in \O_{E_{\p}}$ of $\mathcal{D}_{E_{\p}/F_{\q}}$ satisfies $\nu_{\p}(e)=1$ by \cite[Chapter 3, Proposition 13]{Serre2}. By \cite[Chapter 3, Proposition 8]{Serre2}, we have the following equality of ideals of $\O_{E_{\p}}$ 
\[
\mathcal{D}_{E_{\p}/\Q_p}=\mathcal{D}_{E_{\p}/F_{\q}}\mathcal{D}_{F_{\q}/\Q_p},
\]
where $\mathcal{D}_{F_{\q}/\Q_p}$ is the different of $F_{\q}$ over $\Q_p$. 
Let $f \in \O_{F_{\q}}$ be a generator of $\mathcal{D}_{F_{\q}/\Q_p}$. Then we have
\[
\nu_{\p}(d^{-1})=\nu_{\p}(e)+\nu_{\p}(f)=1+\nu_{\p}(f).
\]
Since $d^{-1}, f \in \O_{F_{\q}}$ and $E_{\p}/F_{\q}$ is totally ramified, the integers $\nu_{\p}(d^{-1})$ and $\nu_{\p}(f)$ are even. This is absurd. Therefore $E_{\p}/F_{\q}$ is unramified.

Suppose that $p=2$. We note that the pairing $\langle\ ,\ \rangle_{\p}$ is even, i.e.\ $\langle x, x \rangle_{\p}\in 2\Z_2$ for every $x \in \mathscr{O}_{E_{\p}}$, since the transcendental lattice $T(X_{\C})$ is an even lattice; see \cite[Chapter 1, Proposition 3.5]{Huybrechts} for example. 
We have 
\begin{align*}
\langle x, x \rangle_{\p}&=\mathrm{Tr}_{E_{\p}/\Q_2}(dxc(x))\\
&=\mathrm{Tr}_{E_{\p}/\Q_2}(d\mathrm{Nm}_{E_{\p}/F_{\q}}(x))\\
&=2\mathrm{Tr}_{F_{\q}/\Q_2}(d\mathrm{Nm}_{E_{\p}/F_{\q}}(x)).
\end{align*}
Hence we have
\[
\mathrm{Tr}_{F_{\q}/\Q_2}(d\mathrm{Nm}_{E_{\p}/F_{\q}}(x)) \in \Z_2
\]
for every $x \in \mathscr{O}_{E_{\p}}$.
We assume that $E_{\p}/F_{\q}$ is totally ramified. By Lemma \ref{norm}, we see that $\mathrm{Tr}_{F_{\q}/\Q_2}(dx) \in \Z_2$ for every $x \in \mathscr{O}_{F_{\q}}$. 
Hence we have $\mathcal{D}_{F_{\q}/\Q_2}\subset d^{-1}\mathscr{O}_{F_{\q}}$. Consequently, we have
\[
\mathcal{D}_{F_{\q}/\Q_2}\O_{E_{\p}}\subset d^{-1}\mathscr{O}_{E_{\p}}=\mathcal{D}_{E_{\p}/\Q_2}.
\] 
Since $E_{\p}/F_{\q}$ is totally ramified, we have $\mathcal{D}_{E_{\p}/F_{\q}} \subsetneq \O_{E_{\p}}$. Therefore we have 
\begin{align*}
\mathcal{D}_{E_{\p}/\Q_2}&=\mathcal{D}_{E_{\p}/F_{\q}}\mathcal{D}_{F_{\q}/\Q_2}\\
&\subsetneq \mathcal{D}_{F_{\q}/\Q_2}\O_{E_{\p}}.
\end{align*}
This is absurd. Therefore $E_{\p}/F_{\q}$ is unramified.

The proof of Proposition \ref{local_lemma} is complete.
\end{proof}

\section{Formal Brauer groups of $K3$ surfaces}\label{Formal Brauer group}
In this section, we recall the definition and basic properties of the formal Brauer groups of $K3$ surfaces in characteristic $p>0$.

Let $X_k$ be a $K3$ surface over an algebraically closed field $k$ of characteristic $p>0$. Let 
\[ 
\xymatrix{
 \Phi^2_{X_{k}}\colon {\mathrm{Art}_k}   \ar[r]  &{\text{(Abelian  groups)}}}
 \]
 be the functor defined by 
 \[
 \xymatrix{{R}      \ar@{|->}[r]  &\ker({H^2_{\rm{\acute{e}t}}(X_k\otimes_k{R},\mathbb{G}_m)}\rightarrow{H^2_{\rm{\acute{e}t}}(X_k,\mathbb{G}_m)}),
}
\]
where $\mathrm{Art}_k$ is the category of local artinian $k$-algebras with residue field $k$, and (Abelian  groups) is the category of abelian groups.
The functor $\Phi^2_{X_{k}}$ is pro-representable by a smooth one-dimensional formal group scheme $\hatBr(X_k)$ over $k$ \cite[Chapter II, Corollary 2.12]{Artin-Mazur}. The \textit{height} of the formal Brauer group of $X_k$ is defined to be the height of $\hatBr(X_k)$
\[
h(X_k):=h({\hatBr(X_k)}).
\]
The Picard group $\Pic(X_k)$ is a free $\Z$-module of rank $\leq 22$. The rank 
\[
\rho(X_k):=\rank_{\Z}{\Pic(X_k)}
\]
 is called the \textit{Picard number} of $X_k$. Since $X_k$ is a projective surface, we have $\rho(X_k)\geq1$. When ${h(X_k)}=\infty$, we say $X_k$ is \textit{supersingular}. The Tate conjecture for $K3$ surfaces \cite{Charles13}, \cite{Kim-Madapusi}, \cite{Madapusi}, \cite{Maulik} implies that $X_k$ is supersingular if and only if $\rho(X_k)=22$. (See also \cite[Corollaire 0.5]{Benoist}, \cite[Corollary 17.3.7]{Huybrechts}.) If $h(X_k)\neq\infty$, the inequality 
 \[
 1 \leq \rho(X_k) \leq 22-2h(X_k)
 \] is satisfied \cite[Theorem 0.1]{Artin74}. Therefore, if $h(X_k)\neq\infty$, we have $1\leq{h(X_k)}\leq10$. 

In the rest of this section, we fix a power $q$ of $p$. Let $\F_q$ be the finite field with $q$ elements. Let $X_{\F_q}$ be a $K3$ surface over $\F_q$, and we put $X_{\overline{\F}_q}:=X_{\F_q}\otimes_{\F_q}{\overline{\F}_q}$. The geometric Frobenius element $\sigma \in \mathrm{Gal}(\overline{\F}_q/\F_q)$ acts on the $\'e$tale cohomology
\[
H^2_{\rm{\acute{e}t}}(X_{\overline{\F}_q}, \Q_{\ell}{(1)})
\]
for a prime number $\ell \neq p$. The characteristic polynomial $P(T)$ of $\sigma$ acting on $H^2_{\rm{\acute{e}t}}(X_{\overline{\F}_q}, \Q_{\ell}{(1)})
$ is a monic polynomial of degree $22$ with coefficients in $\Q$. It is independent of the choice of $\ell \neq p$. We consider the roots of $P(T)$ 
\[
\alpha_1, \alpha_2, \dotsc, \alpha_i, \dotsc, \alpha_{22}
\]
as elements of $\overline{\Q}$. 
We fix an embedding $\overline{\Q} \hookrightarrow \overline{\Q}_p$, and let $\nu\colon \overline{\Q}_p\to \Q\cup\lbrace\infty\rbrace$ be the $p$-adic valuation normalized by $\nu(p)=1$. 
The following results are well-known.

\begin{prop}\label{height-Frobenius}
\begin{enumerate}
\item[(1)] The number of $i$ such that 
$\alpha_i$ is a root of unity is equal to $\rho(X_{\overline{\F}_q})$.
\item[(2)] If $h(X_{\overline{\F}_q})<\infty$, then the number of $i$ such that $\nu(\alpha_i)>0$ is equal to $h(X_{\overline{\F}_q})$.
\item[(3)] If $h(X_{\overline{\F}_q})=\infty$, 
then every $\alpha_i$ is a root of unity.
Hence it satisfies $\nu(\alpha_i)=0$. 
\end{enumerate}
\end{prop}
\begin{proof}
The assertion of the Picard number 
follows from the Tate conjecture for $K3$ surfaces 
\cite{Charles13}, \cite{Kim-Madapusi}, \cite{Madapusi}, \cite{Maulik}.
For (2) and (3), see \cite[Proposition 6.17]{Liedtke}, \cite[Theorem 1]{Katz-Messing}.
\end{proof}

\section{Proof of Theorem \ref{supersingular reduction}}\label{Section_proof_thm_1}
In this section, we shall prove Theorem \ref{supersingular reduction}. Similar arguments can be found in Taelman's paper; 
see the proof of \cite[Proposition 25]{Taelman}.

Let $\pi_v$ be a uniformizer of $K_v$. By the local-global compatibility of class field theory, the Artin reciprocity map $\mathrm{Art}_{K}$ sends the element 
\[
x:=(1, 1, \dotsc, 1, \pi_v, 1, \dotsc, 1, 1)\in \A^{\times}_{K}
\]
to $\sigma|_{K^{\mathrm{ab}}}$, where $\sigma \in \mathrm{Gal}(\overline{K_v}/K_v)$ is a lift of the geometric Frobenius element. We want to compute the eigenvalues of $\sigma$ acting on
\begin{align*}
H^2_{\rm{\acute{e}t}}(X_{\overline{K_v}}, \Q_{\ell}{(1)})&\simeq H^{2}(X_{\C}, \Z(1))\otimes_{\Z}{\Q_{\ell}}\\
&\simeq (\Pic(X_{\C})\otimes_{\Z}{\Q_{\ell}})\oplus (T(X_{\C})\otimes_{\Z}{\Q_{\ell}})
\end{align*}
for a prime number $\ell \neq p$. The image of $x$ under the composite of the norm map $\mathrm{Nm}_{\A_K/\A_E}\colon\A^{\times}_{K} \rightarrow \A^{\times}_{E}$ with the projection $\mathrm{pr}\colon \A^{\times}_{E} \rightarrow \A^{\times}_{E, f}$ is the element
\[
y:=(1, 1, \dotsc, 1, \mathrm{Nm}_{K_v/E_{\p}}(\pi_v), 1, \dotsc, 1, 1)\in \A^{\times}_{E, f},
\]
where $\mathrm{Nm}_{K_v/E_{\p}}\colon K^{\times}_v \rightarrow E^{\times}_{\p}$ is the norm map induced by the embedding $E_{\p}\hookrightarrow K_v$ coming from $E \subset K$.

If $\q$ splits in $E$, the image of 
$
y \in\A^{\times}_{E, f}
$
 under the map
\[
\xymatrix{
\A^{\times}_{E, f} \ar[r]&  G(\Af)/G(\Q), &y \mapsto c(y)y^{-1}
}
\]
is represented by 
$$z:=(1, 1, \dotsc, 1, \mathrm{Nm}_{K_v/E_{\p}}(\pi_v)^{-1}, c(\mathrm{Nm}_{K_v/E_{\p}}(\pi_v)), 1, \dotsc, 1, 1)\in G(\Af),$$
where $c \colon E_{\p} \simeq E_{\p'}$ is the isomorphism induced by 
the complex conjugation $c \colon E \to E$ and $\p' (\neq \p)$ is a unique finite place of $E$ above $\q$ other than $\p$.

If $\q$ does not split in $E$, the image of 
$
y \in\A^{\times}_{E, f}
$
 under the map
\[
\xymatrix{
\A^{\times}_{E, f} \ar[r]&  G(\Af)/G(\Q), &y \mapsto c(y)y^{-1}
}
\]
is represented by 
$$z:=(1, 1, \dotsc, 1, c(\mathrm{Nm}_{K_v/E_{\p}}(\pi_v))\mathrm{Nm}_{K_v/E_{\p}}(\pi_v)^{-1}, 1, \dotsc, 1, 1)\in G(\Af),$$
where $c \colon E_{\p} \simeq E_{\p}$ is the involution induced by 
the complex conjugation $c \colon E \to E$. 

In both cases, by the commutative diagram of Theorem \ref{CM main theorem}, there exists an element $\eta \in G(\Q)\subset E^{\times}$ such that 
$$\rho(\sigma|_{K^{\mathrm{ab}}})=z\eta \in G(\Af).$$

Before proving Theorem \ref{supersingular reduction}, we prove the following lemma. We fix an embedding $\overline{\Q} \hookrightarrow \overline{\Q}_p$.
Let $\nu\colon \overline{\Q}_p\to \Q\cup\lbrace\infty\rbrace$ be the $p$-adic valuation normalized by $\nu(p)=1$. 
Let $\Q(\eta) \subset E$ be the subfield generated by $\eta$.
Let $\p_0$ be the finite place of $\Q(\eta)$ below $\p$.

\begin{lem}\label{lemma1}
\begin{enumerate}
\item[(1)] If $\q$ splits in $E$, the element $\eta$ is not a root of unity and the number of conjugations $\xi \in \overline{\Q}$ of $\eta$ over $\Q$ with $\nu(\xi)>0$ is equal to $[\Q(\eta)_{\p_0}:\Q_p]$.  
\item[(2)] If $\q$ does not split in $E$, the element $\eta$ is a root of unity.
\end{enumerate}
\end{lem}
\begin{proof}
At first, since $\mathrm{Gal}(K^{\mathrm{ab}}/K)$ is compact, we observe that $\rho(\sigma|_{K^{\mathrm{ab}}})=z\eta$ lies in the maximal compact subgroup 
$$\{\, x \in \widehat{\Z}\otimes_{\Z}{\mathscr{O}_E} \mid xc(x)=1 \, \}\subset G(\Af).$$

If $\q$ splits in $E$, we see that $\p$ is a unique finite place of $E$ above $p$ such that the valuation of $\eta \in E_{\p}$ is positive. Hence $\p_0$ is a unique finite place of $\Q(\eta)$ such that the valuation of $\eta \in \Q(\eta)_{\p_0}$ is positive. It follows that $\eta$
is not a root of unity and the number of conjugations $\xi \in \overline{\Q}$ of $\eta$ over $\Q$ with $\nu(\xi)>0$ is equal to $[\Q(\eta)_{\p_0}:\Q_p]$.

If $\q$ does not split in $E$, we see that $\eta$
is an algebraic integer. 
Moreover, since we have $\eta c(\eta)=1$ and $E$ is a $CM$ field, 
every conjugation $\xi \in \overline{\Q}\subset \C$ of $\eta$ over $\Q$ has complex absolute value $1$. By Kronecker's theorem, the algebraic integer $\eta$ is a root of unity. \end{proof}

\begin{proof}[\textbf{Proof of Theorem \ref{supersingular reduction}}] 
Since $T({X_{\C}})\otimes_{\Z}{\Q}$ is a one-dimensional $E$-vector space, 
by Theorem \ref{CM main theorem}, 
the characteristic polynomial of $\sigma$ acting on $T(X_{\C})\otimes_{\Z}{\Q_{\ell}}$ can be identified with the characteristic polynomial of
$
E \to E
$,
$x \mapsto x\eta
$
as a $\Q$-linear map. Hence it is equal to
$$f(t)^{[E:\Q(\eta)]},$$ 
where $f(t)$ is the monic minimal polynomial of $\eta \in E$ over $\Q$. 

If $\q$ splits in $E$, no root of $f(t)$ is a root of unity by Lemma \ref{lemma1} (1).
Since $\p$ is a unique finite place of $E$ above $\p_0$, we have $[E:\Q(\eta)]=[E_{\p}:\Q(\eta)_{\p_0}]$ by \cite[Chapter 2, Theorem 1(iii)]{Serre2}. Hence the number of roots $\xi\in \overline{\Q}$ of $f(t)^{[E:\Q(\eta)]}$ with $\nu(\xi)>0$ counted with their multiplicities is equal to $[E_{\p}:\Q_p]$ by Lemma \ref{lemma1} $(1)$. The eigenvalues of $\sigma$ acting on 
\[
\Pic(X_{\C})\otimes_{\Z}{\Q_{\ell}}\simeq\Pic(X_{\overline{K_{v}}})\otimes_{\Z}{\Q_{\ell}}
\] 
are roots of unity. Therefore, the number of eigenvalues $\xi \in \overline{\Q}$ of $\sigma$ acting on $H^2_{\rm{\acute{e}t}}(X_{\overline{K_v}}, \Q_{\ell}{(1)})$ which are roots of unity is equal to the rank of $\Pic(X_{\C})$, counted with their multiplicities. Moreover, the number of eigenvalues $\xi\in \overline{\Q}$ of $\sigma$ acting on $H^2_{\rm{\acute{e}t}}(X_{\overline{K_v}}, \Q_{\ell}{(1)})$ with $\nu(\xi)>0$ is equal to $[E_{\p}:\Q_p]$, counted with their multiplicities. It follows that the Picard number of $\mathscr{X}_{\overline{v}}$ is equal to the Picard number of $X_{\C}$, and the height $h(\mathscr{X}_{\overline{v}})$ of the formal Brauer group of $\mathscr{X}_{\overline{v}}$ is equal to $[E_{\p}:\Q_p]$ by Proposition \ref{height-Frobenius}. 

If $\q$ does not split in $E$, every root of $f(t)$ is a root of unity by Lemma \ref{lemma1} (2). Therefore, every eigenvalue of $\sigma$ acting on $H^2_{\rm{\acute{e}t}}(X_{\overline{K_v}}, \Q_{\ell}{(1)})$ is a root of unity. It follows that 
the Picard number of $\mathscr{X}_{\overline{v}}$ is equal to $22$, 
and the height $h(\mathscr{X}_{\overline{v}})$ of the formal Brauer group of $\mathscr{X}_{\overline{v}}$ is $\infty$ by Proposition \ref{height-Frobenius}. 

The proof of Theorem \ref{supersingular reduction} is complete.
\end{proof}
\section{Integral $p$-adic $\mathrm{\acute{E}}$tale cohomology of $K3$ surfaces with $CM$}
\label{The Galois module}
In this section and the next section, we make some preparations to prove Theorem \ref{Artin invariant}. In this section, we shall describe the Galois module $H^2_{\rm{\acute{e}t}}(X_{\overline{K_v}}, \Z_{p})$. We keep the notation in previous sections. We assume that $X_{\C}$ is a projective $K3$ surface over $\C$ with $CM$ by $E$ satisfying the assumptions of Theorem \ref{Artin invariant}. In particular, we assume that 
$X_{\C}$ has a model $X_K$ over a number field $K \subset \C$ 
containing $E$, $X_K$ has good reduction at $v$, and $\q$ does not split in $E$. By Theorem \ref{supersingular reduction}, the reduction $\mathscr{X}_{\overline{v}}$ is a supersingular $K3$ surface over $\overline{k(v)}$.

We fix a uniformizer $\pi_{\p}$ of $E_\p$. We recall the Lubin-Tate character attached to an embedding 
$f\colon E_\p \hookrightarrow K_v.$ We have an isomorphism 
$$\mathrm{Gal}(E^{\mathrm{ab}}_{\p}/E_{\p})\simeq \mathscr{O}^{\times}_{E_{\p}}\times\widehat{\Z}$$
induced by the (fixed) uniformizer $\pi_{\p}$ and the Artin reciprocity map provided by local class field theory with Deligne's normalization. We fix an isomorphism $\overline{E_{\p}}\simeq\overline{K_v}$ which extends $f$. This identification induces a homomorphism
\[
\mathrm{Gal}(\overline{K_v}/K_v)\rightarrow \mathrm{Gal}(E^{\mathrm{ab}}_{\p}/E_{\p}).
\]
We denote this homomorphism by the same symbol $f$.
We define a continuous homomorphism  $\chi_{\pi_{\p}, f}\colon \mathrm{Gal}(\overline{K_v}/K_v) \rightarrow \mathscr{O}^{\times}_{E_{\p}}$ attached to $\pi_{\p}$ and the embedding $ f\colon E_\p \hookrightarrow K_v$ by the following composite:
\[
\xymatrix{
\chi_{\pi_{\p}, f}\colon \mathrm{Gal}(\overline{K_v}/K_v)\ar[r]_-{f}&\mathrm{Gal}(E^{\mathrm{ab}}_{\p}/E_{\p})\simeq \mathscr{O}^{\times}_{E_{\p}}\times\widehat{\Z} \ar[r]_-{\mathrm{pr}}&\mathscr{O}^{\times}_{E_{\p}}.
}
\]
It does not depend on the choice of an isomorphism $\overline{E_{\p}}\simeq\overline{K_v}$ extending $f$.
Since we will not change the uniformizer $\pi_{\p}$ in the sequel, we will abbreviate
$\chi_{f}:=\chi_{\pi_{\p}, f}$.

We denote the $\mathrm{Gal}(\overline{K_v}/K_v)$-module ${\O_{E_{\p}}}$ arising from a character $\chi\colon \mathrm{Gal}(\overline{K_v}/K_v) \rightarrow \mathscr{O}^{\times}_{E_{\p}}$ by
\[
(\O_{E_{\p}}, \chi).
\]

\begin{rem}\label{Lubin-Tate module}
\rm{Let $\mathcal{F}_{\pi_{\p}}$ be the \textit{Lubin-Tate formal group} over $\O_{E_{\p}}$ associated with $\pi_{\p}$.
(This is denoted by $\mathcal{G}$ in \cite[Section 2.2]{AGHM}.)
For an embedding $f\colon E_\p \hookrightarrow K_v$, the $p$-adic Tate module of the base change $\mathcal{F}_{\pi_{\p}, \O_{K_v}}$ of $\mathcal{F}_{\pi_{\p}}$ along the induced homomorphism $\O_{E_\p} \hookrightarrow \O_{K_v}$ is isomorphic to $(\O_{E_{\p}}, \chi_f)$ as a $\mathrm{Gal}(\overline{K_v}/K_v)$-module.
See \cite{Lubin-Tate} for details. (See also \cite[Chapter 3, Appendix, Proposition 4]{Serre3}.)
For this reason, we will often call $\chi_f$ the \textit{Lubin-Tate character} attached to $\pi_{\p}$ and $f$ (or simply the Lubin-Tate character attached to $f$).}
\end{rem}

Let 
\[
\chi_{\mathrm{cyc}}\colon \mathrm{Gal}(\overline{K_v}/K_v) \rightarrow \Z^{\times}_p
\]
be the $p$-adic cyclotomic character. For an integer $n \in \Z$, a free $\Z_p$-module of rank $1$ on which $\mathrm{Gal}(\overline{K_v}/K_v)$ acts via $\chi^{n}_{\mathrm{cyc}}$ is denoted by $\Z_p(n)$. We put $\Z_p:=\Z_p(0)$. Recall that we have an embedding $E \subset K$ as subfields of $\C$. Let 
\[
\iota \colon E_{\p} \hookrightarrow K_{v} 
\]
be the embedding induced by $E \subset K$. 
Let $\chi_{\iota}$ be the Lubin-Tate character 
attached to the embedding $\iota \colon E_{\p} \hookrightarrow K_v$. 
Similarly, let $\chi_{\iota\circ c}$ be the Lubin-Tate character 
attached to the composite of the involution $c \colon E_{\p}\simeq E_{\p}$ 
(it is induced by the complex conjugation $c$ on $E$) 
with $\iota \colon E_{\p} \hookrightarrow K_v$.
We define the character 
\[
\widetilde{\chi}\colon \mathrm{Gal}(\overline{K_v}/K_v) \rightarrow \mathscr{O}^{\times}_{E_{\p}}
\] 
by 
\[
\widetilde{\chi}:=\chi^{-1}_{\iota}\chi_{\iota\circ c}\chi^{-1}_{\mathrm{cyc}}.
\]
Here the composite of $\chi_{\mathrm{cyc}}$ with the canonical embedding $\Z^{\times}_p\hookrightarrow\mathscr{O}^{\times}_{E_{\p}}$, $x \mapsto x$ is denoted by the same symbol $\chi_{\mathrm{cyc}}$.
The character $\widetilde{\chi}$ is independent of the choice of $\pi_\p$; see the proof of Proposition \ref{decomposition} below.

Recall that $X_{K_v}$ has good reduction and the reduction $\mathscr{X}_{\overline{v}}$ is a supersingular $K3$ surface over $\overline{k(v)}$.
To prove Theorem \ref{Artin invariant}, we may replace $K$ by a finite extension of it.
Hence we may assume 
$$\Pic(X_{\C})\simeq\Pic(X_K)\simeq\Pic(X_{K_v})$$
and 
$$\Pic(\mathscr{X}_{\overline{v}})\simeq\Pic(\mathscr{X}\otimes_{\mathscr{O}_{K_v}}k(v)).$$
Under these assumptions, the action of $\mathrm{Gal}(\overline{K_v}/K_v)$ on $H^2_{\rm{\acute{e}t}}(X_{\overline{K_v}}, \Z_{p})$ is described as follows.

\begin{prop}\label{decomposition}
Under the above assumptions, there is an isomorphism of $\mathrm{Gal}(\overline{K_v}/K_v)$-modules 
\[
H^2_{\rm{\acute{e}t}}(X_{\overline{K_v}}, \Z_{p})\simeq \Z_p(-1)^{\oplus22-[E_{\p}:\Q_p]} \oplus ({\O_{E_{\p}}}, \widetilde{\chi}).
\]
\end{prop}

\begin{proof}
We shall show there is an isomorphism of $\mathrm{Gal}(\overline{K_v}/K_v)$-modules 
$$H^2_{\rm{\acute{e}t}}(X_{\overline{K_v}}, \Z_{p}(1))\simeq \Z_p^{\oplus22-[E_{\p}:\Q_p]} \oplus ({\O_{E_{\p}}}, \chi^{-1}_{\iota}\chi_{\iota\circ c}). $$
By the first assumption of Theorem \ref{Artin invariant}, the discriminant of $\Pic(X_{\C})$ is not divisible by $p$. Hence we have 
$$H^2_{\rm{\acute{e}t}}(X_{\overline{K_v}}, \Z_{p}(1))\simeq (\Pic(X_{\C})\otimes_{\Z}{\Z_{p}})\oplus (T(X_{\C})\otimes_{\Z}{\Z_{p}}).$$
Since we have assumed $\Pic(X_{\C})\simeq\Pic(X_K)\simeq\Pic(X_{K_v})$, the action of $\mathrm{Gal}(\overline{K_v}/K_v)$ on $\Pic(X_{\C})\otimes_{\Z}{\Z_{p}}$ is trivial. 

Hence it is enough to show there is an isomorphism of $\mathrm{Gal}(\overline{K_v}/K_v)$-modules 
$$T(X_{\C})\otimes_{\Z}{\Z_{p}} \simeq \Z^{\oplus22-m}_p\oplus (\O_{E_{\p}}, \chi^{-1}_{\iota}\chi_{\iota\circ c}),$$  
where we put $m:=\rank_{\Z}{\Pic(X_{\C})}+[E_{\p}:\Q_p]$. Let
\[
\rho\colon \mathrm{Gal}(K^{\mathrm{ab}}/K) \rightarrow G(\Af)
\]
be the continuous homomorphism provided by the main theorem of complex multiplication for $K3$ surfaces; see Theorem \ref{CM main theorem}.
Since $\mathrm{Gal}(K^{\mathrm{ab}}_v/K_v)$ is compact, the composite 
\[
\xymatrix{
\mathrm{Gal}({K^{\mathrm{ab}}_v}/K_v)\ar[r]& \mathrm{Gal}(K^{\mathrm{ab}}/K) \ar[r]_-{\rho}& G(\Af)\ar[r]_-{\mathrm{pr}}&(E\otimes_{\Q}{\Q_p})^{\times}
}
\]
factors through a continuous homomorphism 
\[
\rho_p \colon \mathrm{Gal}(K_v^{\mathrm{ab}}/K_v) \rightarrow (\mathscr{O}_{E}\otimes_{\Z}{\Z_p})^{\times}.
\]

By the second assumption of Theorem \ref{Artin invariant}, we have
\[
\mathrm{End}_{\mathrm{Hdg}}(T(X_{\C}))\otimes_{\Z}{\Z_{(p)}}\simeq \mathscr{O}_{E}\otimes_{\Z}{\Z_{(p)}}.
\]
By Lemma \ref{quadratic space}, we see that $T(X_{\C})\otimes_{\Z}{\Z_{p}}$ is a free $\mathscr{O}_{E}\otimes_{\Z}{\Z_{p}}$-module of rank $1$. There is a decomposition \[
\mathscr{O}_{E}\otimes_{\Z}{\Z_p} \simeq \prod_{\p'|p}\mathscr{O}_{E_{\p'}},
\]
where $\p'$ is a finite place of $E$ above $p$. Hence we have an isomorphism of $\mathrm{Gal}(\overline{K_v}/K_v)$-modules 
\[
T(X_{\C})\otimes_{\Z}{\Z_{p}} \simeq  \prod_{\p'|p}\mathscr{O}_{E_{\p'}},
\]
where the action of $\mathrm{Gal}(\overline{K_v}/K_v)$ on $\mathscr{O}_{E_{\p'}}$ is given by $\mathrm{pr}_{\p'}\circ \rho_p$. Here 
\[
\mathrm{pr}_{\p'}\colon (\mathscr{O}_{E}\otimes_{\Z}{\Z_p})^{\times} \simeq \prod_{\p'|p}\mathscr{O}^{\times}_{E_{\p'}} \rightarrow \mathscr{O}^{\times}_{E_{\p'}}
\] is the projection. Therefore, it suffices to prove that
$$(\mathrm{pr}_{\p'}\circ \rho_p)(g)=
\begin{cases}
  1 & (\p' \neq \p)\\
  \chi^{-1}_{\iota}(g)\chi_{\iota\circ c}(g)  & (\p' = \p) 
\end{cases}
$$
for every $g \in \mathrm{Gal}(K_v^{\mathrm{ab}}/K_v)$, where $\chi_{\iota}$ and $\chi_{\iota\circ c}$ are considered as continuous homomorphisms from $\mathrm{Gal}(K_v^{\mathrm{ab}}/K_v)$ to ${\O}^{\times}_{E_{\p}}$. 

It suffices to prove these equalities when we restrict them to $K_v^{\times}$ along the local Artin reciprocity map 
\[
\mathrm{Art}_{K_v} \colon K_v^{\times} \hookrightarrow \mathrm{Gal}(K_v^{\mathrm{ab}}/K_v)
\]
since $\mathrm{Art}_{K_v}(K_v^{\times}) \subset \mathrm{Gal}(K_v^{\mathrm{ab}}/K_v)$ is dense. We take an element $x \in K_v^{\times}$. By the local-global compatibility of class field theory and the commutative diagram of Theorem \ref{CM main theorem}, there exists an element $\eta \in G(\Q)\subset E^{\times}$ such that 
$$\rho(\mathrm{Art}_{K_v}(x))=(1, 1, \dotsc, 1, c(\mathrm{Nm}_{K_v/E_{\p}}(x))\mathrm{Nm}_{K_v/E_{\p}}(x)^{-1}, 1, \dotsc, 1, 1)\eta.$$
Since $\Pic(\mathscr{X}_{\overline{v}})\simeq\Pic(\mathscr{X}\otimes_{\mathscr{O}_{K_v}}k(v))$ and $\mathscr{X}_{\overline{v}}$ is supersingular by Theorem \ref{supersingular reduction}, the action of $\mathrm{Gal}(\overline{K_v}/K_v)$ on $T(X_{\C})\otimes_{\Z}{\Q_{\ell}}$ is trivial for all $\ell \neq p$. It follows that $\eta=1$. Hence we have 
$$(\mathrm{pr}_{\p'}\circ \rho_p)(\mathrm{Art}_{K_v}(x))=
\begin{cases}1 & (\p' \neq \p)\\  
c(\mathrm{Nm}_{K_v/E_{\p}}(x))\mathrm{Nm}_{K_v/E_{\p}}(x)^{-1} & (\p'=\p).
\end{cases}
$$

Recall that we have fixed a uniformizer $\pi_{\p}$ of $E_{\p}$. It gives an isomorphism $\mathscr{O}^{\times}_{E_{\p}}\times{\Z} \simeq E^{\times}_{\p}, (x, n) \mapsto x\pi^{n}_{\p}$. The composite of the inverse map of this isomorphism with the projection is denoted by $\chi' \colon E^{\times}_{\p} \rightarrow \O^{\times}_{E_{\p}}$. Then, we have 
\[
\begin{cases}
\chi_{\iota}(\mathrm{Art}_{K_v}(x))=\chi'(\mathrm{Nm}_{K_v/E_{\p}}(x))\\
\chi_{\iota\circ c}(\mathrm{Art}_{K_v}(x))=\chi'(c(\mathrm{Nm}_{K_v/E_{\p}}(x))).
\end{cases}
\]
We note that, for every $z_1, z_2 \in E^{\times}_{\p}$ which have the same valuation, we have
\[
\chi'(z_1)\chi'(z_2)^{-1}=\chi'(z_1z^{-1}_2)=z_1z^{-1}_2.
\]
Hence we have
\begin{align*}
\chi^{-1}_{\iota}(\mathrm{Art}_{K_v}(x))\chi_{\iota\circ c}(\mathrm{Art}_{K_v}(x))
&=\chi'(c(\mathrm{Nm}_{K_v/E_{\p}}(x)))\chi'(\mathrm{Nm}_{K_v/E_{\p}}(x))^{-1}\\
&=c(\mathrm{Nm}_{K_v/E_{\p}}(x))\mathrm{Nm}_{K_v/E_{\p}}(x)^{-1}.
\end{align*}

This completes the proof of Proposition \ref{decomposition}.
\end{proof}

\section{The Breuil-Kisin modules associated with $K3$ surfaces with $CM$}\label{F-crystal associated with Lubin-Tate characters} 

In this section, we give an explicit description of the $F$-crystal $H^2_{\mathrm{cris}}(\mathscr{X}_{\overline{v}}/W)$ of the supersingular $K3$ surface $\mathscr{X}_{\overline{v}}$ over $\overline{k(v)}$, where $W:=W(\overline{\F}_{p})$ is the ring of Witt vectors of $\overline{k(v)}=\overline{\F}_{p}$. We keep the notation in Section \ref{The Galois module}.
In particular, we have fixed a uniformizer $\pi_{\p}$ of $E_{\p}$.

Let $K_v^{\mathrm{ur}}$ be the maximal unramified extension of $K_v$ in $\overline{K_v}$, and 
\[
L:=\widehat{K_v^{\mathrm{ur}}}
\] the completion of $K_v^{\mathrm{ur}}$. We identify $W$ with the $p$-adic completion of the maximal unramified extension of $\Z_p$ in $\overline{K_v}$. Hence we have an embedding $W\hookrightarrow L$. Then $L$ is a finite totally ramified extension of $\mathrm{Frac}(W)$.

Let $E_{\p, 0} \subset E_{\p}$ be the maximal unramified extension of $\Q_p$ in $E_{\p}$.
We put $d:=[E_{\p, 0}:\Q_p]$.
By Proposition \ref{local_lemma}, the quadratic extension $E_{\p}/F_{\q}$ is unramified.
It follows that $d$ is an even integer.

We denote the set of embeddings from $E_{\p, 0}$ to $\mathrm{Frac}(W)\subset L$ by $\mathrm{Emb}(E_{\p, 0})$. Let $\iota_0 \in \mathrm{Emb}(E_{\p, 0})$ be an embedding induced by 
\[
\xymatrix{E_{\p, 0} \ar@{^{(}->}[r]& E_{\p} \ar@{^{(}->}[r]_-{\iota}& K_{v} \ar@{^{(}->}[r]& L.}
\]
We denote the Frobenius on $\mathrm{Frac}(W)$ by $\varphi$. Then we have 
$$\mathrm{Emb}(E_{\p, 0})=\{\, \varphi^{i}\circ\iota_0 \mid 0 \leq i \leq d-1 \}.$$ 
When we consider $\O_{E_{\p, 0}}$ as a subring of $\O_{E_{\p}}$ by the inclusion $\O_{E_{\p, 0}} \subset \O_{E_{\p}}$ and a subring of $W$ by the embedding $\varphi^{i}\circ\iota_0$, we denote it by $R_i$. 
We put 
\[
W_i:=\O_{E_{\p}}\otimes_{R_i}{W}.
\]
Then we have a decomposition 
\[
\O_{E_{\p}}\otimes_{\Z_p}W \simeq\underset{0\leq i \leq d-1}{\prod}W_i.
\]

The main result of this section is as follows:
\begin{thm}\label{crystalline cohomology1}
The $F$-crystal $H^2_{\mathrm{cris}}(\mathscr{X}_{\overline{v}}/W)$ is isomorphic to the $F$-crystal 
$$W(-1)^{\oplus22-[E_{\p}:\Q_p]}\oplus(\O_{E_{\p}}\otimes_{\Z_p}W, \beta).$$
Here, the $F$-crystal $(\O_{E_{\p}}\otimes_{\Z_p}W, \beta)$ is a $W$-module $\O_{E_{\p}}\otimes_{\Z_p}W$ equipped with an $\O_{E_{\p}}$-equivariant Frobenius $\varphi$ given by $\varphi(1)=\beta$, and the element
\[
\beta\in \O_{E_{\p}}\otimes_{\Z_p}W\simeq \underset{0\leq i \leq d-1}{\prod}W_i
\] is specified as follows; the $i$-th component $\beta_{i}$ of $\beta$ is  
\[
\beta_{i}=
\begin{cases}
p\pi_{\p} &(i=1)\\
p\pi^{-1}_{{\p}} &(i=d')\\
p &(i \neq 1, d'),
\end{cases}
\]
where we put $d' := d/2+1$ if $d \neq 2$, 
and $d' := 0$ if $d=2$.
\end{thm}

To prove Theorem \ref{crystalline cohomology1}, we use integral comparison theorem recently established by Bhatt, Morrow, and Scholze, which says that it is enough to calculate the Breuil-Kisin module
 $$
 \mathfrak{M}(H^2_{\rm{\acute{e}t}}(X_{\overline{K_v}}, \Z_{p}))
 $$
  associated with the Galois module $H^2_{\rm{\acute{e}t}}(X_{\overline{K_v}}, \Z_{p})$ to calculate the $F$-crystal $H^2_{\mathrm{cris}}(\mathscr{X}_{\overline{v}}/W)$; see Theorem \ref{crystalline cohomology BMS} for details.
 We compute this Breuil-Kisin module $\mathfrak{M}(H^2_{\rm{\acute{e}t}}(X_{\overline{K_v}}, \Z_{p}))$ using an explicit description of the Breuil-Kisin modules associated with Lubin-Tate characters proved by Andreatta, Goren, Howard, and Madapusi Pera; see Proposition \ref{BK-modules ass with Lubin-Tate} for details.

First, we recall the definition of Breuil-Kisin modules from \cite[Section 1.2]{Kisin}.
We fix a uniformizer $\pi_L \in L$, and a system $\{ \pi^{{1/p}^n}_L \}_{n\geq 0} \subset \overline{L}$ $(i \geq0)$ of $p^n$-th roots of $\pi_L$ such that $(\pi^{{1/p}^{n+1}}_L)^p=\pi^{{1/p}^n}_L$ for $n \geq0$.
Let $E(u) \in W[u]$ be the monic minimal polynomial of $\pi_L$ over $\mathrm{Frac}(W)$, 
i.e.\ it is the (monic) Eisenstein polynomial for $\pi_L$. We define a Frobenius $\varphi$ on 
$$\mathfrak{S}:=W[[u]]$$
 which acts on $W$ as the canonical Frobenius and sends $u$ to $u^p$.
 A \textit{Breuil-Kisin module} (over $\mathscr{O}_L$ with respect to $\{ \pi^{{1/p}^n}_L \}_{n \geq 0}$) is a free $\mathfrak{S}$-module $\mathfrak{M}$ of finite rank equipped with an isomorphism of $\mathfrak{S}$-modules
$$\varphi \colon (\varphi^{*}\mathfrak{M})[1/E(u)]\simeq\mathfrak{M}[1/E(u)].$$ 
Let $\mathscr{O}_{\mathcal{E}}$ be the $p$-adic completion of $\mathfrak{S}[1/u]$. This is a Cohen ring whose residue field is $\overline{\F}_p((u))$. We fix a separable closure $\overline{\F}_p((u))^{\mathrm{sep}}$ of $\overline{\F}_p((u))$. Let $\mathscr{O}_{\mathcal{E}_{\mathrm{ur}}}$ be the strict henselization of $\mathscr{O}_{\mathcal{E}}$ with respect to $\overline{\F}_p((u))^{\mathrm{sep}}$, and $\mathscr{O}_{\widehat{\mathcal{E}_{\mathrm{ur}}}}$ the $p$-adic completion of $\mathscr{O}_{\mathcal{E}_{\mathrm{ur}}}$. 
The Frobenius $\varphi$ on $\mathfrak{S}$ extends naturally to $\mathscr{O}_{\widehat{\mathcal{E}_{\mathrm{ur}}}}$.

We put
\[
L_{\infty}:=\bigcup_{n\geq0}L(\pi^{{1/p}^n}_L).
\]
By \cite[Corollaire 3.2.3]{Wintenberger}, there is an isomorphism 
\[
\mathrm{Gal}(\overline{L}/L_{\infty})\simeq\mathrm{Gal}(\overline{\F}_p((u))^{\mathrm{sep}}/\overline{\F}_p((u))).
\]
We have an action of $\mathrm{Gal}(\overline{L}/L_{\infty})$ on $\mathscr{O}_{\widehat{\mathcal{E}_{\mathrm{ur}}}}$ via this isomorphism.
\begin{thm}[Kisin {\cite{Kisin}}]\label{Kisin functor}
There exists a covariant fully faithful tensor functor 
\[
\Lambda \mapsto \mathfrak{M}(\Lambda)
\]
from the category of $\Z_p$-lattices in crystalline $\mathrm{Gal}(\overline{L}/L)$-representations to the category of Breuil-Kisin modules. The Breuil-Kisin module $\mathfrak{M}(\Lambda)$ is characterized by the existence of an isomorphism of $\mathscr{O}_{\widehat{\mathcal{E}_{\mathrm{ur}}}}$-modules
$$\Lambda\otimes_{\Z_p}{\mathscr{O}_{\widehat{\mathcal{E}_{\mathrm{ur}}}}} \simeq \mathfrak{M}(\Lambda)\otimes_{\mathfrak{S}}{\mathscr{O}_{\widehat{\mathcal{E}_{\mathrm{ur}}}}}$$
compatible with Frobenius morphisms and the action of $\mathrm{Gal}(\overline{L}/L_{\infty})$, where
\begin{itemize}
\item the Frobenius morphism on the left-hand side (resp.\ right-hand side) is given by $1\otimes\varphi$ (resp.\ $\varphi\otimes\varphi$), and
\item the action of an element $g \in \mathrm{Gal}(\overline{L}/L_{\infty})$ on the left-hand side (resp.\ right-hand side) is given by $g\otimes g$ (resp.\ $1\otimes g$).
\end{itemize}
\end{thm}
\begin{proof}
See \cite[Theorem 1.2.1]{Kisin}. The characterization follows from the construction of the functor. See also \cite[Theorem 4.4]{BMS}.  
\end{proof}

We shall give an explicit description of the Breuil-Kisin module $\mathfrak{M}(H^2_{\rm{\acute{e}t}}(X_{\overline{K_v}}, \Z_{p}))$. By Proposition \ref{decomposition}, we have an isomorphism of $\mathrm{Gal}(\overline{L}/L)$-modules $$H^2_{\rm{\acute{e}t}}(X_{\overline{K_v}}, \Z_{p})\simeq \Z_p(-1)^{\oplus22-[E_{\p}:\Q_p]} \oplus ({\O_{E_{\p}}}, \widetilde{\chi}).$$ 
Hence we have an isomorphism of Breuil-Kisin modules
$$\mathfrak{M}(H^2_{\rm{\acute{e}t}}(X_{\overline{K_v}}, \Z_{p}))\simeq \mathfrak{M}(\Z_p(-1))^{\oplus22-[E_{\p}:\Q_p]} \oplus \mathfrak{M}(({\O_{E_{\p}}}, \widetilde{\chi})).$$ 
 
The Breuil-Kisin module $\mathfrak{M}(\Z_p(-1))$ is described as follows. It is identified with $\mathfrak{S}$ as an $\mathfrak{S}$-module, and the Frobenius $\varphi$ is defined by $\varphi(1)=pE(u)/E(0)$. 
Thus, to describe $\mathfrak{M}(H^2_{\rm{\acute{e}t}}(X_{\overline{K_v}}, \Z_{p}))$, it is enough to compute $\mathfrak{M}(({\O_{E_{\p}}}, \widetilde{\chi}))$.

We have a decomposition 
\[
\O_{E_{\p}}\otimes_{\Z_p}\mathfrak{S}\simeq\underset{0\leq i \leq d-1}{\prod}\mathfrak{S}_i,
\]
where we put 
\[
\mathfrak{S}_i:=W_i[[u]]
\]
for $0\leq i \leq d-1$. Let 
\[
E_0(u) \in W_0[u]
\] be the monic minimal polynomial of $\pi_L$ over $\mathrm{Frac}(W_0)$, where we embed $W_0$ into $L$ by the embedding
\[
\xymatrix{
W_0=\O_{E_{\p}}\otimes_{R_0}{W}\ar[r]&L, &{x\otimes y \mapsto \iota(x)y}. 
}
\]

By Proposition \ref{local_lemma}, the involution $c$ is non-trivial on $E_{\p, 0}$. Hence the embedding $\iota\circ{c} \colon E_{\p}\hookrightarrow L$
induces the embedding
$
\varphi^{d/2}\circ\iota_0 \in \mathrm{Emb}(E_{\p, 0}).
$
Let 
\[
E_{d/2}(u) \in W_{d/2}[u]
\] be the monic minimal polynomial of $\pi_L$ over $\mathrm{Frac}(W_{d/2})$, where we embed $W_{d/2}$ into $L$ by the embedding

\[
\xymatrix{
W_{d/2}=\O_{E_{\p}}\otimes_{R_{d/2}}{W}\ar[r]&L, &{x\otimes y \mapsto (\iota\circ{c})(x)y}.
}
\]

For a character $\chi\colon \mathrm{Gal}(\overline{K_v}/K_v) \rightarrow \mathscr{O}^{\times}_{E_{\p}}$, let $({\O_{E_{\p}}}, \chi)^{\vee}$ be a $\mathrm{Gal}(\overline{K_v}/K_v)$-module
\[ 
\O^{\vee}_{E_{\p}}:=\mathrm{Hom}_{\Z_p}(\O_{E_{\p}}, \Z_p)
\]
induced by $({\O_{E_{\p}}}, \chi)$.

We shall recall an explicit description of the Breuil-Kisin module $\mathfrak{M}(({\O_{E_{\p}}}, \chi_{\iota})^{\vee})$
(resp.\ $\mathfrak{M}(({\O_{E_{\p}}}, \chi_{\iota\circ c})^{\vee})$)
associated with $({\O_{E_{\p}}}, \chi_{\iota})^{\vee}$ (resp.\ (${\O_{E_{\p}}}, \chi_{\iota\circ c})^{\vee}$).
We remark that the Breuil-Kisin module $\mathfrak{M}(({\O_{E_{\p}}}, \chi_{\iota})^{\vee})$
(resp.\ $\mathfrak{M}(({\O_{E_{\p}}}, \chi_{\iota\circ c})^{\vee})$) admits an $\O_{E_{\p}}$-action
such that the Frobenius commutes with the $\O_{E_{\p}}$-action since $\mathfrak{M}(-)$ is a functor.

For an element
$
\gamma \in \O_{E_{\p}}\otimes_{\Z_p}\mathfrak{S},
$ we define a Frobenius semi-linear map $\varphi_\gamma$ on $\O_{E_{\p}}\otimes_{\Z_p}\mathfrak{S}$ which commutes with the $\mathscr{O}_{E_{\p}}$-action by $\varphi_\gamma(1)=\gamma$.  We denote the $\mathfrak{S}$-module $\O_{E_{\p}}\otimes_{\Z_p}\mathfrak{S}$ equipped with the Frobenius semi-linear map $\varphi_\gamma$ by
\[
(\O_{E_{\p}}\otimes_{\Z_p}\mathfrak{S}, \gamma).
\]

\begin{prop}[Andreatta, Goren, Howard, Madapusi Pera {\cite[Proposition 2.2.1]{AGHM}}]\label{BK-modules ass with Lubin-Tate}
\quad
\begin{enumerate}
\item[(1)] The Breuil-Kisin module 
\[
\mathfrak{M}(({\O_{E_{\p}}}, \chi_{\iota})^{\vee})
\]
 is $\O_{E_{\p}}$-equivariantly isomorphic to the Breuil-Kisin module $(\O_{E_{\p}}\otimes_{\Z_p}\mathfrak{S}, \beta_{\iota})$, where the element 
 \[
 \beta_{\iota} \in \O_{E_{\p}}\otimes_{\Z_p}\mathfrak{S}\simeq\underset{0\leq i \leq d-1}{\prod}\mathfrak{S}_i
 \] is specified as follows; the $i$-th component $\beta_{\iota, i}$ of $\beta_{\iota}$ is 
$$
\beta_{\iota, i}=
\begin{cases}
\pi_{{\p}}E_0(u)/E_0(0) &(i=0)\\
1 &(i \neq 0).
\end{cases}
$$

\item[(2)] The Breuil-Kisin module 
\[
\mathfrak{M}(({\O_{E_{\p}}}, \chi_{\iota\circ c})^{\vee})
\]
 is $\O_{E_{\p}}$-equivariantly isomorphic to the Breuil-Kisin module $(\O_{E_{\p}}\otimes_{\Z_p}\mathfrak{S}, \beta_{\iota\circ c})$, where the element  
\[
\beta_{\iota\circ c}\in \O_{E_{\p}}\otimes_{\Z_p}\mathfrak{S}\simeq\underset{0\leq i \leq d-1}{\prod}\mathfrak{S}_i
\] is specified as follows; the $i$-th component $\beta_{\iota\circ c, i}$ of $\beta_{\iota\circ c}$ is 
\[
\beta_{\iota\circ c, i}=
\begin{cases}
\pi_{{\p}}E_{d/2}(u)/E_{d/2}(0) &(i=d/2)\\
1 &(i \neq d/2).
\end{cases}
\]
\end{enumerate}
\end{prop}

\begin{proof}
See \cite[Proposition 2.2.1]{AGHM}. We note that the polynomial $\mathcal{E}_{\iota_0}(u) \in W_{\iota_0}[u]$ in the notation of \cite[Section 2]{AGHM} is a minimal polynomial of $\varpi$ whose constant term is equal to a fixed uniformizer $\iota_0(\pi_{E})$.
In our notation, the polynomial $\mathcal{E}_{\iota_0}(u)$ is translated to $\pi_{{\p}}E_0(u)/E_0(0)$ in (1) and $\pi_{{\p}}E_{d/2}(u)/E_{d/2}(0)$ in (2).
\end{proof}

\begin{prop}\label{BK-module chi(-1)}
The Breuil-Kisin module $\mathfrak{M}(({\O_{E_{\p}}}, \widetilde{\chi}))$ is $\O_{E_{\p}}$-equivariantly isomorphic to the Breuil-Kisin module $(\O_{E_{\p}}\otimes_{\Z_p}\mathfrak{S}, \widetilde{\beta})$, where the element
\[
\widetilde{\beta}\in \O_{E_{\p}}\otimes_{\Z_p}\mathfrak{S}\simeq\underset{0\leq i \leq d-1}{\prod}\mathfrak{S}_i\]
 is specified as follows; the $i$-th component $\widetilde{\beta}_{i}$ of $\widetilde{\beta}$ is 
\[
\widetilde{\beta}_{i}=
\begin{cases}
p(E(u)/E(0))(\pi_{\p}E_0(u)/E_0(0)) &(i=0)\\
p(E(u)/E(0))(\pi_{{\p}}E_{d/2}(u)/E_{d/2}(0))^{-1} &(i=d/2)\\
pE(u)/E(0) &(i \neq 0, d/2).
\end{cases}
\]

\end{prop}
\begin{proof}
By an isomorphism of $\mathrm{Gal}(\overline{K_v}/K_v)$-modules
$$(\O_{E_{\p}}, \chi^{-1}_{\iota})[1/p] \simeq ({\O_{E_{\p}}}, \chi_{\iota})^{\vee}[1/p]$$
induced by the trace map $\mathrm{Tr}_{E_{\p}/\Q_p}\colon E_{\p} \rightarrow \Q_p$,
the $\O_{E_{\p}}$-lattice $\O^{\vee}_{E_{\p}}$ on the right-hand side is translated to the $\O_{E_{\p}}$-lattice
$
\pi^{-d}_{\p}\O_{E_{\p}} \subset (\O_{E_{\p}}, \chi^{-1}_{\iota})[1/p]
$
for an integer $d \geq 0$.
It follows that there exists an isomorphism of $\mathrm{Gal}(\overline{K_v}/K_v)$-modules
\[
(\O_{E_{\p}}, \chi^{-1}_{\iota}) \simeq ({\O_{E_{\p}}}, \chi_{\iota})^{\vee}.
\]
Thus, the $\mathrm{Gal}(\overline{L}/L)$-module $(\O_{E_{\p}}, \widetilde{\chi})$ is isomorphic to the $\mathrm{Gal}(\overline{L}/L)$-module 
\[
(\O_{E_{\p}}, \chi_{\iota\circ c}{\chi^{-1}_{\mathrm{cyc}}}) \otimes_{\O_{E_{\p}}}({\O_{E_{\p}}}, \chi_{\iota})^{\vee}.\]

We shall show that the Breuil-Kisin module $\mathfrak{M}((\O_{E_{\p}}, \widetilde{\chi}))$ associated with $(\O_{E_{\p}}, \widetilde{\chi})$ is isomorphic to the tensor product
\[
\mathfrak{M}:=\mathfrak{M}((\O_{E_{\p}}, \chi_{\iota\circ c}{\chi^{-1}_{\mathrm{cyc}}}))\otimes_{\O_{E_{\p}}\otimes_{\Z_p}\mathfrak{S}}\mathfrak{M}(({\O_{E_{\p}}}, \chi_{\iota})^{\vee}).\]
Since $\mathfrak{M}(({\O_{E_{\p}}}, \chi_{\iota})^{\vee})$ is a free $\O_{E_{\p}}\otimes_{\Z_p}\mathfrak{S}$-module of rank $1$ by Proposition \ref{BK-modules ass with Lubin-Tate} (1), the tensor product $\mathfrak{M}$ is a free $\mathfrak{S}$-module of finite rank.
Thus, we see that $\mathfrak{M}$ is a Breuil-Kisin module.
By the characterization of $\mathfrak{M}((\O_{E_{\p}}, \widetilde{\chi}))$ as in Theorem \ref{Kisin functor}, we have $\mathfrak{M}((\O_{E_{\p}}, \widetilde{\chi})) \simeq \mathfrak{M}$.

Using an isomorphism $\O_{E_{\p}}\simeq{\O^{\vee}_{E_{\p}}}$ as above, Proposition \ref{BK-modules ass with Lubin-Tate} (2), and the fact that the functor $\mathfrak{M}(-)$ is a tensor functor, we see that the Breuil-Kisin module $$\mathfrak{M}((\O_{E_{\p}}, \chi_{\iota\circ c}{\chi^{-1}_{\mathrm{cyc}}}))$$ is $\O_{E_{\p}}$-equivariantly isomorphic to the  Breuil-Kisin module $(\O_{E_{\p}}\otimes_{\Z_p}\mathfrak{S}, \beta')$, where the element
\[
\beta'\in\O_{E_{\p}}\otimes_{\Z_p}\mathfrak{S}\simeq\underset{0\leq i \leq d-1}{\prod}\mathfrak{S}_i
\] 
is specified as follows; the $i$-th component $\beta'_{i}$ of $\beta'$ is 
\[
\beta'_{i}=
\begin{cases}
p(E(u)/E(0))(\pi_{{\p}}E_{d/2}(u)/E_{d/2}(0))^{-1} &(i=d/2)\\
pE(u)/E(0) &(i \neq d/2).
\end{cases}
\]
By using this description, the assertion follows from the isomorphism $\mathfrak{M}((\O_{E_{\p}}, \widetilde{\chi})) \simeq \mathfrak{M}$ and Proposition \ref{BK-modules ass with Lubin-Tate} (1).
\end{proof}

Now, Theorem \ref{crystalline cohomology1} follows from Proposition \ref{BK-module chi(-1)} and the following integral comparison theorem established by Bhatt, Morrow, and Scholze.

\begin{thm}[Bhatt, Morrow, Scholze {\cite[Theorem 1.2]{BMS0}, \cite[Theorem 14.6]{BMS}}]\label{crystalline cohomology BMS}
There is an isomorphism of $F$-crystals
$$H^2_{\mathrm{cris}}(\mathscr{X}_{\overline{v}}/W)\simeq \varphi^{*}(\mathfrak{M}(H^2_{\rm{\acute{e}t}}(X_{\overline{K_v}}, \Z_{p}))/u\mathfrak{M}(H^2_{\rm{\acute{e}t}}(X_{\overline{K_v}}, \Z_{p}))).$$
\end{thm}
\begin{proof}
See \cite[Theorem 14.6]{BMS}. (See also \cite[Theorem 1.2]{BMS0}.) Note that we can apply \cite[Theorem 14.6]{BMS} for the proper smooth algebraic space $\mathscr{X}$ over $\Spec\mathscr{O}_{K_v}$ since the special fiber is a scheme; see \cite[Theorem 2.4]{CCL} for details.
\end{proof}

\section{Proof of Theorem \ref{Artin invariant}}\label{Section_proof_thm2}
In this section, we shall prove Theorem \ref{Artin invariant}. We use the same notation as in Section \ref{The Galois module} and Section \ref{F-crystal associated with Lubin-Tate characters}. 
Recall that $\pi_{\p}$ is a fixed uniformizer of $E_{\p}$, 
and we put $d:=[E_{\p, 0}:\Q_p]$. By Proposition \ref{local_lemma}, the quadratic extension $E_{\p}/F_{\q}$ is unramified. Hence we have 
\[
d=[k(\p):\F_p]=2[k(\q):\F_p].
\]

\begin{proof}[\textbf{Proof of Theorem \ref{Artin invariant}}]
The Chern class map 
\[
{\mathrm{ch_{cris}}}\colon \Pic(\mathscr{X}_{\overline{v}}){\otimes}_{\Z}{W} \rightarrow H^2_{\mathrm{cris}}(\mathscr{X}_{\overline{v}}/W)
\]
is injective and preserves pairings. Since the Artin invariant $a$ satisfies $\disc\Pic(\mathscr{X}_{\overline{v}})=-p^{2a}$, it is equal to the length of the cokernel of $\mathrm{ch_{cris}}$ as a $W$-module. We shall show that it is equal to $[k(\q):\F_p]=[k(\p):\F_{p}]/2=d/2$. 

By Theorem \ref{crystalline cohomology1}, we have an isomorphism of $F$-crystals
$$H^2_{\mathrm{cris}}(\mathscr{X}_{\overline{v}}/W)\simeq W(-1)^{\oplus22-[E_{\p}:\Q_p]}\oplus(\O_{E_{\p}}\otimes_{\Z_p}W, \beta),$$ 
where the element $\beta \in \O_{E_{\p}}\otimes_{\Z_p}W$ is taken as in Theorem \ref{crystalline cohomology1}. By \cite[Corollary 1.6]{Ogus}, we have
 $$\Pic(\mathscr{X}_{\overline{v}})\otimes_{\Z}{\Z_p}\simeq H^2_{\mathrm{cris}}(\mathscr{X}_{\overline{v}}/W)^{\varphi=p}.$$ 
Hence the length of the cokernel of the following injection
\[
H^2_{\mathrm{cris}}(\mathscr{X}_{\overline{v}}/W)^{\varphi=p}\otimes_{\Z_p}{W} \rightarrow H^2_{\mathrm{cris}}(\mathscr{X}_{\overline{v}}/W)
\]
as a $W$-module is equal to the Artin invariant $a$.

Recall that $W(-1)$ is the $F$-crystal of rank $1$ whose underlying $W$-module is $W$ and the Frobenius $\varphi$ is given by $\varphi(1)=p$. Hence $1 \in W$ is a $\Z_p$-basis of $W(-1)^{\varphi=p}$, and the canonical map  
\[
W(-1)^{\varphi=p}\otimes_{\Z_p}{W}\rightarrow W
\]
is an isomorphism. Therefore, the Artin invariant $a$ is equal to the length of the cokernel of the following injection
\[
(\O_{E_{\p}}\otimes_{\Z_p}W, \beta)^{\varphi=p}\otimes_{\Z_p}{W} \rightarrow (\O_{E_{\p}}\otimes_{\Z_p}W, \beta)
\]
as a $W$-module.

We shall compute a $\Z_p$-basis of $(\O_{E_{\p}}\otimes_{\Z_p}W, \beta)^{\varphi=p}$.
An element 
$
x \in\O_{E_{\p}}\otimes_{\Z_p}W\simeq\underset{0\leq i \leq d-1}{\prod}W_i
$
satisfies $\varphi(x)=px$ if and only if 
the $i$-th components $x_i \in W_i$ ($0 \leq i \leq d-1$) of $x$ satisfy 
$$
\begin{cases}
\pi_{{\p}}\varphi(x_0)=x_1\\
\pi^{-1}_{{\p}}\varphi(x_{d/2})=x_{d/2+1}\\
\varphi(x_i)=x_{i+1} &(i \neq 0, d/2),\\
\end{cases}
$$
where we put $x_d:=x_0$. 
From these equations, we have $\varphi^d(x_{d/2+1})=x_{d/2+1}$ and hence 
$
x_{d/2+1} \in \O_{E_\p} \subset W_{d/2+1},
$
where we put $W_d:=W_0$.
Conversely, if $x_{d/2+1}$ is an element of $\O_{E_\p} \subset W_{d/2+1}$, then for the elements $x_i \in W_i$ $(0 \leq i \leq d-1)$ determined by the above equations, the element $x=(x_0, \dotsc, x_{d-1})$ is in $(\O_{E_{\p}}\otimes_{\Z_p}W, \beta)^{\varphi=p}$.
Thus, the following composite
\[
\O_{E_{\p}}\otimes_{\Z_p}W \simeq \underset{0\leq i \leq d-1}{\prod}W_i \to W_{d/2+1},
\]
where the second map is the projection, maps
$(\O_{E_{\p}}\otimes_{\Z_p}W, \beta)^{\varphi=p}$
isomorphically onto $\O_{E_\p}$.
We denote by $f$ the induced isomorphism of $\Z_p$-modules
$(\O_{E_{\p}}\otimes_{\Z_p}W, \beta)^{\varphi=p} \simeq \O_{E_\p}$.
The following composite
\[
\xymatrix{
\O_{E_{\p}}\otimes_{\Z_p}W \ar[r]_-{f^{-1}\otimes \mathrm{id}_W}& (\O_{E_{\p}}\otimes_{\Z_p}W, \beta)^{\varphi=p} \otimes_{\Z_p}W \ar[r]_-{}& \O_{E_{\p}}\otimes_{\Z_p}W
}
\]
is equal to the $W$-linear map
$
g_{\pi}\colon \O_{E_{\p}}\otimes_{\Z_p}W \to \O_{E_{\p}}\otimes_{\Z_p}W
$,
$x \mapsto \pi{x}$.
Here
$
\pi \in \O_{E_{\p}}\otimes_{\Z_p}W\simeq\underset{0\leq i \leq d-1}{\prod}W_i
$
is the element whose $i$-th component $\pi_{i}$ is 
$$\pi_{i}=
\begin{cases}
1 & (i=0)\\
\pi_{{\p}} & (1 \leq i \leq d/2)\\
 1 & (d/2+1 \leq i \leq d-1). 
\end{cases}
$$

Therefore, to show Theorem \ref{Artin invariant}, it suffices to show that the length $N$ of the cokernel of $g_\pi$
as a $W$-module is equal to $d/2.$ The length of the cokernel of the injection 
$
W_i \to W_i
$, 
$x \mapsto \pi_{\p}{x}
$
as a $W$-module is equal to $1$ since the extension $\mathrm{Frac}(W_i)/\mathrm{Frac}(W)$ is totally ramified and $\pi_{\p} \in W_i$ is a uniformizer of $\mathrm{Frac}(W_i)$.
From this, we have 
\[
N=d/2=[k(\q):\F_p].
\]

The proof of Theorem \ref{Artin invariant} is complete.\end{proof}

\section{$K3$ surfaces with Picard number $20$}\label{singular K3}
In this section, we study the relation between our results and the results of Shimada \cite{Shimada}, \cite{Shimada2}. 

Let $X_{\C}$ be a projective $K3$ surface over $\C$ whose Picard number is $20$. Such $K3$ surfaces are also called \textit{singular} $K3$ surfaces. The $K3$ surface $X_{\C}$ has $CM$ by an imaginary quadratic field $E$; see \cite[Chapter 3, Remark 3.10]{Huybrechts}. Hence $X_{\C}$ has a model $X_{K}$ over a number field $K \subset \C$ which contains the image of 
\[
\epsilon_{X_{\C}}\colon E(X_{\C})\rightarrow \mathrm{End}_{\C}(H^{2, 0}(X_{\C}))\simeq \C.
\] 

Shimada showed the following results for $K3$ surfaces with Picard number $20$ and $p \neq 2$. 
\begin{thm}[Shimada]\label{Shimada1}
\begin{enumerate}
\item[(1)] Let $p \neq 2$ be an odd prime number, and $v$ a finite place of $K$ whose residue characteristic is $p$. Assume that $p$ does not divide $\disc\Pic(X_{\C})$, the $K3$ surface $X_K$ has good reduction at $v$, and the geometric special fiber $\mathscr{X}_{\overline{v}}$ is supersingular. Then the Artin invariant of $\mathscr{X}_{\overline{v}}$ is $1$, and $p$ does not split in $E$.
\item[(2)] There exists a finite set $S$ of prime numbers containing all the prime numbers dividing $\disc\Pic(X_{\C})$ which satisfies the following property: for every odd prime number $p \notin \{ 2 \} \cup S$ and every finite place $v$ of $K$ whose residue characteristic is $p$, if the $K3$ surface $X_K$ has good reduction at $v$, then the geometric special fiber $\mathscr{X}_{\overline{v}}$ is supersingular if and only if $p$ does not split in $E$. 
\end{enumerate}
 \end{thm}
 \begin{proof}
 See \cite[Proposition 5.5]{Shimada} and \cite[Theorem 1]{Shimada2}.
 \end{proof}
 \begin{rem}
 \rm{Shimada states above results using the Legendre symbol $\displaystyle \bigg( \frac{\disc\Pic(X_{\C})}{p} \bigg).$ We shall show that for an odd prime number $p\neq2$, $\displaystyle \bigg( \frac{\disc\Pic(X_{\C})}{p} \bigg)=-1$ if and only if $p$ does not split in $E$; see Lemma \ref{singular4} below.}
 \end{rem}
 
We shall prove our results imply Shimada's results, and Shimada's results also hold in the case $p=2$. More precisely, we shall prove the following result.
\begin{thm}\label{Shimada2}
Let $p$ be a prime number. Let $v$ be a finite place of $K$ whose residue characteristic is $p$. Assume that $p$ does not divide $\disc \Pic(X_{\C})$, and the $K3$ surface $X_K$ has good reduction at $v$. Then the geometric special fiber $\mathscr{X}_{\overline{v}}$ is supersingular if and only if $p$ does not split in $E$. Moreover, if $\mathscr{X}_{\overline{v}}$ is supersingular, its Artin invariant is $1$. 
\end{thm}
\begin{proof}
The assertions follow from Theorem \ref{supersingular reduction}, Theorem \ref{Artin invariant}, and Proposition \ref{singular3} below.
\end{proof}
Before stating Proposition \ref{singular3}, we make some preparations. Let $p$ be a prime number not dividing $\disc\Pic(X_{\C})$. Let 
\[\left(
\begin{array}{cc}
a_1 & a_2   \\
a_2 & a_3
\end{array}
\right) 
\]
be the intersection matrix of $T(X_{\C})$ with respect to a $\Z$-basis $e_1, e_2$. Since $T(X_{\C})$ is an even positive definite lattice, we have $a_1\neq0$ and $a_3 \neq 0$, and we can write $a_1$ and $a_3$ in the following form:
\[
\begin{cases}
a_1=2p^n{a'_1} \\
a_3=2p^n{a'_3}
\end{cases}
\]
for some $n \geq0$ such that at least one of $a'_1$ or $a'_3$ is not divisible by $p$. Exchanging $e_1, e_2$ if necessary, we may assume $a'_1$ is not divisible by $p$.
We shall show that the assumption that $\disc\Pic(X_{\C})$ is not divisible by $p$ implies the conditions of Theorem \ref{Artin invariant}.

\begin{prop}\label{singular3}
Let $X_{\C}$ be a projective $K3$ surface over $\C$ whose Picard number is $20$. Assume that $X_{\C}$ has $CM$ by an imaginary quadratic field $E$. Let $p$ be a prime number not dividing $\disc\Pic(X_{\C})$. Then the following assertions hold.
\begin{itemize}\item The second condition
$$\mathrm{End}_{\mathrm{Hdg}}(T(X_{\C}))\otimes_{\Z}{\Z_{(p)}}\simeq \mathscr{O}_{E}\otimes_{\Z}{\Z_{(p)}}$$ 
of Theorem \ref{Artin invariant} is satisfied, and
\item there is an isomorphism of $\Z_{(p)}$-algebras
\[
\mathscr{O}_{E}\otimes_{\Z}{\Z_{(p)}}\simeq \Z_{(p)}[T]/(a'_1T^2+a_2T+p^{2n}a'_3).
\]
\end{itemize}
\end{prop}
\begin{proof}
Take a unique element $\alpha \in \C \mathop{\backslash} \Q$ such that
$$\alpha e_1+ e_2 \in T(X_{\C})\otimes_{\Z}{\C}\simeq T^{1, -1}\oplus T^{-1, 1}$$ 
is a generator of $T^{1, -1}$ as a $\C$-vector space.
Since $T^{1, -1}$ is an isotropic line, we have
\[
2p^n{a'_1}\alpha^2+2a_2\alpha+2p^n{a'_3}=0.
\]
Hence we have $[\Q(\alpha):\Q]=2$.
We take $f \in \mathrm{End}_{\mathrm{Hdg}}(T(X_{\C}))\otimes_{\Z}{\Z_{(p)}}$ and put 
$$
\begin{cases}
f(e_1)=xe_1+ye_2\\
f(e_2)=ze_1+we_2
\end{cases}
$$
for some $x, y, z, w \in \Z_{(p)}.$ Since $f$ preserves the Hodge structure on $T(X_{\C})\otimes_{\Z}{\Z_{(p)}}$, there exists $\gamma \in \C$ such that
$f(\alpha e_1+ e_2)=\gamma(\alpha e_1+ e_2).$ 
We have 
\begin{align*}
f(\alpha e_1+ e_2)&=\alpha(xe_1+ye_2)+ze_1+we_2\\
&=(x\alpha +z)e_1+(y\alpha +w)e_2.
\end{align*}
Comparing the coefficients of $e_1$ and $e_2$, we have $x\alpha+z=\gamma\alpha $ and $y\alpha+w=\gamma$.
According to the first equality, the homomorphism $\epsilon_{X_{\C}}$ maps $E$ into $\Q(\alpha)$, and hence induces an isomorphism $E \simeq \Q(\alpha)$.

We put $\beta:=p^{n}\alpha$. Then $\beta$ satisfies
\[
a'_1{\beta}^2+a_2\beta+p^{2n}a'_3=0.
\]
In particular, we have $\beta \in \O_{\Q(\alpha)}\otimes_{\Z}\Z_{(p)}$.
Multiplying $p^{2n}$ on the both hand sides of
$
x\alpha+z=y\alpha^2+w\alpha,
$
we have 
\begin{align*}
p^{n}{x}\beta+p^{2n}z&=y\beta^2+p^{n}{w}\beta\\
&=-a_2(a'_1)^{-1}y\beta-p^{2n}a'_3(a'_1)^{-1}y+p^{n}{w}\beta.
\end{align*}
Comparing the coefficients of $1$ and $\beta$, we have
 \[
\begin{cases}
p^{n}{x}=-a_2(a'_1)^{-1}y+p^{n}{w}\\
z=-a'_3(a'_1)^{-1}y.
\end{cases}
\]
We claim that $\epsilon_{X_{\C}}$ induces an isomorphism
\[
\epsilon_{X_{\C}}\colon\mathrm{End}_{\mathrm{Hdg}}(T(X_{\C}))\otimes_{\Z}{\Z_{(p)}} \simeq \Z[\beta]\otimes_{\Z}{\Z_{(p)}}.
\] 
If $n=0$, this follows from the above equalities. If $n>0$, then $a_2$ is not divisible by $p$ since $\disc\Pic(X_{\C})$ is not divisible by $p$. Therefore, from the first equality above, we see that $y$ has of the form 
$y=p^{n}y'$
for some $y' \in \Z_{(p)}$. 
We have $\epsilon_{X_{\C}}(f)=\gamma=y'\beta+w\in \Z[\beta]\otimes_{\Z}{\Z_{(p)}}$.
Hence the following morphism is induced:
\[
\epsilon_{X_{\C}}\colon\mathrm{End}_{\mathrm{Hdg}}(T(X_{\C}))\otimes_{\Z}{\Z_{(p)}}\rightarrow \Z[\beta]\otimes_{\Z}{\Z_{(p)}}. 
\]
From above equalities, we see that this morphism is surjective. Hence it is an isomorphism.

To complete the proof of Proposition \ref{singular3}, we need to show $\Z[\beta]\otimes_{\Z}{\Z_{(p)}}=\O_{\Q(\alpha)}\otimes_{\Z}{\Z_{(p)}}$. Take $s+t\beta \in \O_{\Q(\alpha)}\otimes_{\Z}{\Z_{(p)}}$ $(s, t \in \Q)$. We want to show $s, t \in \Z_{(p)}$. Since $s+t\beta$ is integral over $\Z_{(p)}$, there exists a polynomial $$g(T)=T^2+b_1T+b_2 \in \Z_{(p)}[T]$$ such that $g(s+t\beta)=0$. 
Using $a'_1{\beta}^2+a_2\beta+p^{2n}a'_3=0$ and $g(s+t\beta)=0$, we have 
\[
\begin{cases}
s^2-p^{2n}a'_3(a'_1)^{-1}t^2+b_1s+b_2=0\\
2st-a_2(a'_1)^{-1}t^2+b_1t=0.
\end{cases}
\]
From the second equality, we have $t=0$ or $2s-a_2(a'_1)^{-1}t+b_1=0$.

 If $t=0$, we see that $s \in \Q$ is integral over $\Z_{(p)}$. Hence we have $s \in \Z_{(p)}$. 
 
 We assume $2s-a_2(a'_1)^{-1}t+b_1=0$. From the first equality, we have 
 \[
(a_2(a'_1)^{-1}t-b_1)^2-4p^{2n}a'_3(a'_1)^{-1}t^2+2b_1(a_2(a'_1)^{-1}t-b_1)+4b_2=0. \]
Hence we have
 \[
 \frac{{a^2_2}-4p^{2n}a'_1a'_3}{a'^2_1}t^2+(\mathrm{terms\ of\ degree\ \leq 1\ with\ coefficients\ in}\ \Z_{(p)})=0.
 \]
Since 
 $
 \disc\Pic(X_{\C})={a^2_2}-4p^{2n}a'_1a'_3
 $ and it is not divisible by $p$, we conclude that $t \in \Q$ is integral over $\Z_{(p)}$. Hence we have $t \in \Z_{(p)}.$ It follows that $s=(s+t\beta)-t\beta \in \Q$ is integral over $\Z_{(p)}$. Hence we have $s \in \Z_{(p)}$.
 
The proof of Proposition \ref{singular3} is complete.
\end{proof}

We note that Proposition \ref{singular3} does not hold in general for $K3$ surfaces over $\C$ with $CM$ whose Picard number is less than $20$. (For counterexamples, see Example \ref{counter}.)

In the rest of this section, we study the relation between the condition that $p$ does not split in $E$ and the structure of the lattice $\Pic(X_{\C})$.
\begin{lem}\label{singular4}
Let $p\neq 2$ be an odd prime number which does not divide $\disc\Pic(X_{\C})$. Then, $\displaystyle\bigg( \frac{\disc\Pic(X_{\C})}{p} \bigg)=-1$ if and only if $p$ does not split in $E$. 
\end{lem} 
\begin{proof}
By Proposition \ref{singular3}, we have an isomorphism
\[
\mathscr{O}_{E}\otimes_{\Z}{\Z_{(p)}}\simeq \Z_{(p)}[T]/(a'_1T^2+a_2T+p^{2n}a'_3).
\]
Since $p \neq 2$, by changing the variable $T$, 
we have an isomorphism
\[
\mathscr{O}_{E}\otimes_{\Z}{\Z_{(p)}}\simeq \Z_{(p)}[T]/(T^2-({a^2_2}-4p^{2n}a'_1a'_3)).
\]
Since $\disc\Pic(X_{\C})={a^2_2}-4p^{2n}a'_1a'_3$, we see that
$
\bigg( \frac{\disc\Pic(X_{\C})}{p} \bigg)=-1
$
if and only if $p$ does not split in $E$.\end{proof}

In the case $p=2$, we always have $\displaystyle\bigg( \frac{\disc\Pic(X_{\C})}{2} \bigg)=1$ if $\disc\Pic(X_{\C})$ is not divisible by $2$.  Alternatively, we have the following result.

\begin{lem}\label{singular5}
Assume that $\disc\Pic(X_{\C})$ is not divisible by $2$. Then the following conditions are equivalent:
\begin{itemize}
\item $n=0$ and $a'_3$ is not divisible by $2$. 
\item $2$ does not split in $E$.
\end{itemize}\end{lem}
\begin{proof}
Since $\disc\Pic(X_{\C})$ is not divisible by $2$, we have
\[
\mathscr{O}_{E}\otimes_{\Z}{\Z_{(2)}}\simeq \Z_{(2)}[T]/(a'_1T^2+a_2T+2^{2n}a'_3)
\]
by Proposition \ref{singular3}. Note that $a_2$ is not divisible by $2$ since $\disc\Pic(X_{\C})$ is not divisible by $2$. 
We see that $2$ does not split in $E$ if and only if $2^{2n}a'_3$ is not divisible by $2$.
Hence the assertion follows. \end{proof}

\section{Examples}\label{examples}
In this section, we show that there are Kummer surfaces over $\C$ with $CM$ which do not satisfy the assumptions and the conclusion of Theorem \ref{Artin invariant}.
We also recall some basic results on Kummer surfaces with $CM$.

Let $V$ be a $\Q$-vector space with a $\Q$-Hodge structure.
Let 
$\mathbb{S}:=\mathrm{Res}_{\C/\R}\mathbb{G}_{m}
$
be the Deligne torus, and 
$
h\colon \mathbb{S}\rightarrow \mathrm{GL}(V)_{\R}
$
the homomorphism corresponding the $\Q$-Hodge structure on $V$.
The \textit{Mumford-Tate group} of $V$
is defined as the smallest algebraic $\Q$-subgroup ${\mathrm{MT}}(V)$ of $\GL(V)$ such that
$h(\mathbb{S}(\R)) \subset {\mathrm{MT}}(V)(\R)$.

\begin{rem}\label{MT and Hdg}
	\rm{
	\begin{enumerate}
	\item[(1)] In \cite[Section 3]{Deligne900}, the definition of the Mumford-Tate group of $V$ is given in a slightly different manner.
	It is easy to see that the Mumford-Tate group of $V$ in the sense of \cite[Section 3]{Deligne900} is commutative if and only if the Mumford-Tate group ${\mathrm{MT}}(V)$ of $V$ given here is so.
	\item[(2)] Let $\mathbb{U}_1$ be the kernel of the norm character $\mathrm{Nm}\colon \mathbb{S} \to \mathbb{G}_m$. The \textit{Hodge group} of $V$
is defined as the smallest algebraic $\Q$-subgroup ${\mathrm{Hdg}}(V)$ of $\GL(V)$ such that
$h(\mathbb{U}_1(\R)) \subset {\mathrm{Hdg}}(V)(\R)$.
The following basic facts are used freely in this paper:
\begin{itemize}
	\item ${\mathrm{Hdg}}(V)$ is commutative if and only if ${\mathrm{MT}}(V)$ is so.
	\item For an integer $n$ and the Tate twist $V(n):=V\otimes_{\Q}\Q(n)$, we have ${\mathrm{Hdg}}(V) \simeq {\mathrm{Hdg}}(V(n))$.
	\item If the $\Q$-Hodge structure on $V$ is of weight $0$, we have $\mathrm{MT}(V) \simeq {\mathrm{Hdg}}(V)$.
\end{itemize}
\end{enumerate}}
\end{rem}

Recall that an abelian variety $A$ over $\C$ has \textit{complex multiplication} $(CM)$ if the Mumford-Tate group ${\mathrm{MT}}(H^1(A, \Q))$ of the $\Q$-Hodge structure on the first singular cohomology $H^1(A, \Q)$ is commutative. If $A$ is a simple abelian variety over $\C$, then $A$ has $CM$ if and only if 
\[
\mathrm{End}(A)\otimes_{\Z}{\Q}
\]
is a commutative field over which $H^1(A, \Q)$ has dimension $1$. Moreover, in this case, the field $\mathrm{End}(A)\otimes_{\Z}{\Q}$ is a $CM$ field; see \cite[Proposition 5.1]{Deligne900}.
Every abelian variety $A$ over $\C$ is isogenous to a product of simple abelian varieties, and $A$ has $CM$ if and only if each simple factor has $CM$. It follows that an abelian variety $A$ over $\C$ has $CM$ if and only if there is a $\Q$-algebra $E$ which is a product of $CM$ fields and a homomorphism
\[
i\colon E \rightarrow \mathrm{End}(A)\otimes_{\Z}{\Q}
\]
such that $H^1(A, \Q)$ is a free $E$-module of rank $1$ via $i$. For such $E$, we say $A$ has $CM$ by $E$. 

We give two lemmas on abelian varieties with $CM$, which are presumably well-known to specialists.
\begin{lem}\label{H1H2}
An abelian variety $A$ over $\C$ with $\dim(A) \geq 2$ has $CM$ if and only if the Mumford-Tate group ${\mathrm{MT}}(H^2(A, \Q))$ of the $\Q$-Hodge structure on the second singular cohomology $H^2(A, \Q)$ is commutative. 
\end{lem}
\begin{proof}
Since the proof is standard, we omit the details.
We only remark that the kernel of the homomorphism
$$\mathrm{GL}(H^1(A, \Q))\rightarrow \mathrm{GL}(\wedge^{2}H^1(A, \Q)) \simeq \mathrm{GL}(H^2(A, \Q))$$
is $\{ \pm 1 \}$ since $\dim_{\Q}H^1(A, \Q)=2\dim(A)  \geq 3$.
\end{proof}

Let $A$ be an abelian surface over $\C$. We denote the Kummer surface associated with $A$ by $\mathrm{Km}(A)$. (For the construction and basic properties of Kummer surfaces, see \cite[Chapter 1, Example 1.3 (iii)]{Huybrechts}.)

\begin{lem}\label{Kummer_CM}
An abelian surface $A$ over $\C$ has $CM$ if and only if the Kummer surface $\mathrm{Km}(A)$ is a $K3$ surface with $CM$.
\end{lem}
\begin{proof}
By Zarhin's results on the Hodge group of a $K3$ surface \cite[Theorem 2.2.1]{Zarhin}, \cite[Theorem 2.3.1]{Zarhin}, it follows that a $K3$ surface $X$ over $\C$ has $CM$ if and only if the Mumford-Tate group
$\mathrm{MT}(T(X)\otimes_{\Z}\Q)$ is commutative.
(See also \cite[Proposition 8]{Taelman}.)

There is an isomorphism of $\Q$-Hodge structures
$$H^2(A, \Q(1))\oplus\bigoplus_{1\leq i \leq 16}\Q[Z_i] \simeq H^2(\mathrm{Km}(A),\Q(1)), $$ 
where $[Z_i]$ are the classes of non-singular rational curves $Z_i \subset \mathrm{Km}(A)$ arising from the $2$-torsion points on $A$; see \cite[Lemme 3]{Beauville}. 
It follows that 
$$\mathrm{MT}(T(\mathrm{Km}(A))\otimes_{\Z}\Q)=\mathrm{MT}(H^2(A, \Q(1))).$$
We achieve the proof of Lemma \ref{Kummer_CM} from this equality and Lemma \ref{H1H2}.
\end{proof}

\begin{exm}
\rm{Let $C$ be an elliptic curve over $\C$ with $CM$. The endomorphism algebra 
\[
E(C):=\mathrm{End}(C)\otimes_{\Z}{\Q}
\] 
is an imaginary quadratic field. 
Let $A$ be an abelian surface over $\C$ isogenous to the self-product $C\times_{\C}C$. 
Then the Kummer surface $\mathrm{Km}(A)$ has Picard number $20$, and it has $CM$ by $E(C)$.}
\end{exm}

\begin{exm}\label{CM_C1C2}
\rm{Let $C_1$ and $C_2$ be elliptic curves over $\C$ with $CM$ which are not isogenous to each other. In this case, the imaginary quadratic field $E({C_1})$ is not isomorphic to $E({C_2})$. Hence the $\Q$-algebra 
$E:=E({C_1})\otimes_{\Q}E({C_2})$
is a $CM$ field of degree $4$. Let $A$ be an abelian surface over $\C$ isogenous to $C_1\times_{\C}C_2$. It has $CM$ by $E({C_1})\times E({C_2})$. Let $(E(C_1), \phi_1)$ (resp.\ $(E(C_2), \phi_2)$) be the $CM$ type of $C_1$ (resp.\ $C_2$). By Lemma \ref{Kummer_CM}, the Kummer surface $\mathrm{Km}(A)$ has $CM$. Considering the action of $E({C_1})\times E({C_2})$ on $H^2(A, \Q(1))\simeq H^2(C_1\times_{\C}C_2, \Q(1))$, we see that the image of 
\[
\epsilon_{\mathrm{Km}(A)}\colon E(\mathrm{Km}(A))\rightarrow \mathrm{End}_{\C}(H^{2, 0}(\mathrm{Km}(A)))\simeq \C
\]
contains the elements $\phi_1(x)$ for $x \in E(C_1)$ and $\phi_2(y)$ for $y \in E(C_2)$.
Since the Picard number of $A$ is equal to $2$ by \cite[Chapter IV, Section 20, The Rosati involution, (3)]{Mumford}, the Picard number of $\mathrm{Km}(A)$ is $18$.
Thus, we see that $[E(\mathrm{Km}(A)):\Q]=4$ and $E(\mathrm{Km}(A))$ is isomorphic to $E$.
We conclude that the Kummer surface $\mathrm{Km}(A)$ has $CM$ by $E$ and its Picard number is $18$.}
\end{exm}

\begin{exm}
\rm{Let $A$ be a simple abelian surface over $\C$ with $CM$. The endomorphism algebra $E:=\mathrm{End}(A)\otimes_{\Z}{\Q}$ is a $CM$ field with $[E:\Q]=4$. The Kummer surface $\mathrm{Km}(A)$ has $CM$ by Lemma \ref{Kummer_CM}. The $CM$ type $(E, \Phi)$ of $A$ is primitive \cite[Chapter IV, Section 22, Remark (1)]{Mumford}. We put $\Phi=\{\phi_1, \phi_2 \}$. Recall that the reflex field $E^{*}$ is the subfield of $\C$ generated by the elements $\phi_{1}(x)+\phi_2(x)$ for $x \in E$. Let $E^{*}_0$ be the subfield of $\C$ generated by the elements $\phi_1(x)\phi_2(x)$ for $x \in E$. Since 
\[
\phi_{1}(x)+\phi_2(x)=\phi_1(x+1)\phi_2(x+1)-\phi_1(x)\phi_2(x)-1\in E^{*}_0 ,
\]
we have $E^{*}\subset E^{*}_0$. Conversely, since
\[
2\phi_1(x)\phi_2(x)=(\phi_{1}(x)+\phi_2(x))^{2}-(\phi_{1}(x^2)+\phi_2(x^2))\in E^{*},
\]
we have $E^{*}_0\subset E^{*}.$ Since $[E:\Q]=4$ and the $CM$ type $(E, \Phi)$ is primitive, the reflex field $E^{*}$ satisfies $[E^{*}:\Q]=4$; see \cite[Example 8.4]{Shimura}. Considering the action of $E=\mathrm{End}(A)\otimes_{\Z}{\Q}$ on $H^2(A, \Q(1))$, we see that the image of 
\[
\epsilon_{\mathrm{Km}(A)}\colon E(\mathrm{Km}(A))\rightarrow \mathrm{End}_{\C}(H^{2, 0}(\mathrm{Km}(A)))\simeq \C
\]
contains the elements $\phi_1(x)\phi_2(x)$ for $x \in E$.
Since the Picard number of $A$ is equal to $2$ 
by \cite[Chapter IV, Section 20, The Rosati involution, (3)]{Mumford}, 
the Picard number of $\mathrm{Km}(A)$ is $18$.
Hence we have $[E(\mathrm{Km}(A)):\Q]=4$. 
Therefore the image of $E(\mathrm{Km}(A))$ by $\epsilon_{\mathrm{Km}(A)}$ is equal to $E^{*}$. We conclude that the Kummer surface $\mathrm{Km}(A)$ has Picard number $18$, and it has $CM$ by the reflex field $E^{*}$ of the $CM$ type $(E, \Phi)$ of $A$.}
\end{exm}

Currently, we do not know how to generalize Theorem \ref{Artin invariant} to $K3$ surfaces with $CM$ not satisfying the assumptions of Theorem \ref{Artin invariant}.
We shall give some examples.
In Example \ref{counter2} below, we shall construct a Kummer surface over $\C$ with $CM$ such that it does not satisfy the assumptions of Theorem \ref{Artin invariant} but the conclusion of Theorem \ref{Artin invariant} holds for it.

\begin{exm}\label{counter2}
\rm{Let $E$ be an imaginary quadratic field   such that there is a prime number $p \neq 2$ which is ramified in $E$.
Let $C$ be an elliptic curve over $\C$ with $CM$ by $E$.
We put $A:=C \times_{\C} C$.
As in Remark \ref{elliptic curve reduction}, the elliptic curve $C$ is defined over a number field $K$ containing $E$, it has potential good reduction at a finite place $v$ of $K$ above $p$, and the good reduction of $C$ is supersingular.
It follows that $\mathrm{Km}(A)$ has potential good reduction at $v$, the reduction is supersingular, and the Artin invariant of the reduction is $1$ by \cite[Corollary 7.14]{Ogus}.
By Proposition \ref{local_lemma}, the Kummer surface $\mathrm{Km}(A)$ does not satisfy the assumptions of Theorem \ref{Artin invariant}.
Nevertheless, it satisfies the conclusion of Theorem \ref{Artin invariant}.
}
\end{exm}

Finally, we shall construct a Kummer surface over $\C$ with $CM$ such that it does not satisfy the second assumption of Theorem \ref{Artin invariant} and the conclusion of Theorem \ref{Artin invariant} does not hold for it.

\begin{exm}\label{counter}
\rm{Let $C_1$ and $C_2$ be elliptic curves over $\C$ with $CM$ satisfying $E({C_1})=\Q(\sqrt{-5})$, $\mathrm{End}(C_1)=\O_{\Q(\sqrt{-5})}$, $E({C_2})=\Q(\sqrt{-15})$, and $\mathrm{End}(C_2)=\O_{\Q(\sqrt{-15})}$. We put $E:=\Q(\sqrt{-5}, \sqrt{-15})$ and $A:=C_1\times_{\C}C_2$. By Example \ref{CM_C1C2}, the Kummer surface $\mathrm{Km}(A)$ has $CM$ by $E$. The N$\mathrm{{\acute{e}}}$ron-Severi lattice $NS(A)$ of $A$ is isomorphic to the hyperbolic plane $U$, i.e.\ $U$ is the lattice $U:=\Z{e} \oplus \Z{f}$
whose intersection product is given by $(e, e)=(f, f)=0$ and $(e, f)=1$. Let 
\[
T(A):=NS(A)^{\perp} \subset H^{2}(A, \Z(1)).
\]
be the transcendental lattice of $A$. We have 
\[
T(A)\simeq(H^1(C_1, \Z)\otimes_{\Z} H^1(C_2, \Z))(1).
\]
It is known that there is an isomorphism of $\Z$-Hodge structures
\[
\xymatrix{\phi\colon T(A)\ar[r]^-{\simeq}& T(\mathrm{Km}(A))}
\]
such that $(\phi(x), \phi(y))=2(x, y)$ for every $x, y \in T(A)$; see \cite[Chapter 3, Section 2.5]{Huybrechts}. Therefore, the discriminant $\disc T(\mathrm{Km}(A))$ is not divisible by any odd prime number $p\neq2$. Now, we take $p=5$. The canonical isomorphism
\[
\mathrm{End}_{\Z}(H^1(C_1, \Z))\otimes_{\Z}\mathrm{End}_{\Z}(H^1(C_2, \Z))\simeq\mathrm{End}_{\Z}(H^1(C_1, \Z)\otimes_{\Z} H^1(C_2, \Z))
\]
induces an injection
\[
\mathrm{End}_{\mathrm{Hdg}}(H^1(C_1, \Z))\otimes_{\Z}\mathrm{End}_{\mathrm{Hdg}}(H^1(C_2, \Z))\rightarrow \mathrm{End}_{\mathrm{Hdg}}(H^1(C_1, \Z)\otimes_{\Z} H^1(C_2, \Z))
\]
whose cokernel is torsion-free.
Composing it with the following isomorphisms
\[
\mathrm{End}_{\mathrm{Hdg}}(H^1(C_1, \Z)\otimes_{\Z} H^1(C_2, \Z))\simeq \mathrm{End}_{\mathrm{Hdg}}(T(A))\simeq\mathrm{End}_{\mathrm{Hdg}}(T(\mathrm{Km}(A))),
\]
we have an injective homomorphism
\[
\mathrm{End}_{\mathrm{Hdg}}(H^1(C_1, \Z))\otimes_{\Z}\mathrm{End}_{\mathrm{Hdg}}(H^1(C_2, \Z))\rightarrow \mathrm{End}_{\mathrm{Hdg}}(T(\mathrm{Km}(A)))
\]
whose cokernel is torsion-free. Since the both hand sides are free $\Z$-modules of rank $4$, this injection is an isomorphism. Hence we have
\begin{align*}
\mathrm{End}_{\mathrm{Hdg}}(T(\mathrm{Km}(A)))&\simeq \mathrm{End}_{\mathrm{Hdg}}(H^1(C_1, \Z))\otimes_{\Z}\mathrm{End}_{\mathrm{Hdg}}(H^1(C_2, \Z)) \\
&\simeq\mathrm{End}(C_1)\otimes_{\Z}\mathrm{End}(C_2) \\
&\simeq \O_{\Q(\sqrt{-5})}\otimes_{\Z}\O_{\Q(\sqrt{-15})}.
\end{align*}
On the other hand, we see that $\sqrt{3}$ is in $\O_E\otimes_{\Z}{\Z_{(5)}}$, but $\sqrt{3}=\sqrt{-15}/{\sqrt{-5}}$ is not in $(\O_{\Q(\sqrt{-5})}\otimes_{\Z}\O_{\Q(\sqrt{-15})})\otimes_{\Z}{\Z_{(5)}}$. 
Hence $\mathrm{End}_{\mathrm{Hdg}}(T(\mathrm{Km}(A)))\otimes_{\Z}{\Z_{(5)}}$ is not isomorphic to $\mathscr{O}_{E}\otimes_{\Z}{\Z_{(5)}}$. 
Therefore the Kummer surface $\mathrm{Km}(A)$ does not satisfy the second assumption of Theorem \ref{Artin invariant} for $p = 5$.

As in Remark \ref{elliptic curve reduction}, the elliptic curve $C_1$ (resp.\ $C_2$) is defined over a number field $K$ containing $E$, and it has potential good reduction at a finite place $v$ of $K$ above $5$.
Since $5$ is ramified in $\Q(\sqrt{-5})$ (resp.\ $\Q(\sqrt{-15})$),
the good reduction of $C_1$ (resp.\ $C_2$)  is supersingular.
It follows that $\mathrm{Km}(A)$ has potential good reduction at $v$,  the reduction is supersingular, and the Artin invariant of the reduction is $1$ by \cite[Corollary 7.14]{Ogus}.
Since $5$ is inert in the maximal totally real subfield $F:=\Q(\sqrt{3})$ of $E$, the Kummer surface $\mathrm{Km}(A)$ does not satisfy the conclusion of Theorem \ref{Artin invariant}.
}
\end{exm}
      
\subsection*{Acknowledgements}
The author would like to thank my adviser, Tetsushi Ito, for his kindness, support, and advice. He gave me a lot of invaluable suggestions and pointed out some mistakes in an earlier version of this paper. The author also would like to thank Teruhisa Koshikawa for helpful suggestions on Breuil-Kisin modules and the integral $p$-adic Hodge theory, and Christian Liedtke for helpful comments. 
The author would especially like to thank Yuya Matsumoto for comments on the calculation of the Artin invariants and giving information on the literature on automorphisms of $K3$ surfaces.
Moreover the author would like to thank the anonymous referees for sincere remarks and comments.
The work of the author was supported by Research Fellowships of Japan Society for the Promotion of Science for Young Scientists KAKENHI Grant Number 18J22191.

\end{document}